\documentclass[12pt]{amsart}

\usepackage{amsmath}
\usepackage{amssymb}
\usepackage{array}
\usepackage{enumitem}
\usepackage{hyperref}
\usepackage{verbatim} 
\usepackage{color}

\theoremstyle{plain}
\newtheorem{theorem}{Theorem}[section]
\newtheorem{lemma}[theorem]{Lemma}

\newtheorem{definition}[theorem]{Definition}
\theoremstyle{remark}
\newtheorem{remark}{Remark}[section]
\newtheorem{example}{Example}[section]

\newtheorem*{acknowledgment}{Acknowledgment}

\numberwithin{equation}{section}

\newcommand{\bA}{\mathbb{A}}

\newcommand{\bB}{\mathbb{B}}

\newcommand{\K}{\mathbb{K}}
\newcommand{\bK}{\mathbb{K}}

\newcommand{\R}{\mathbb{R}}
\newcommand{\Z}{\mathbb{Z}}
\newcommand{\C}{\mathbb{C}}

\newcommand{\cT}{\mathcal{T}}
\newcommand{\cS}{\mathcal{S}}

\newcommand{\eff}{\mathbf{f}}
\newcommand{\bff}{\mathbf{f}}
\newcommand{\bfg}{\mathbf{g}}

\newcommand{\bfU}{\mathbf{U}}
\newcommand{\bfV}{\mathbf{V}}
\newcommand{\bfW}{\mathbf{W}}

\newcommand{\Hom}{\mathrm{Hom}}

\newcommand{\id}{\mathrm{id}}

\newcommand{\pr}{\mathrm{pr}}

\newcommand{\set}{\mathrm{Set}}
\newcommand{\Alg}{\mathrm{Alg}}

\newcommand{\sett}[1]{{\{ #1 \}}}
\newcommand{\settt}[1]{{\overline{\{ #1 \}}}}

\newcommand{\inv}{^{-1}}

\newcommand{\ssk}{\smallskip}
\newcommand{\nin}{\noindent}

\newcommand{\ul}{\underline}

\begin{document}

\title[Conceptual Differential Calculus. I]{Conceptual Differential Calculus
\\
Part I: first order local linear algebra}

\author{Wolfgang Bertram}

\address{Institut \'{E}lie Cartan de Lorraine \\
Universit\'{e} de Lorraine at Nancy, CNRS, INRIA \\
Boulevard des Aiguillettes, B.P. 239 \\
F-54506 Vand\oe{}uvre-l\`{e}s-Nancy, France}

\email{\url{wolfgang.bertram@univ-lorraine.fr}}



\begin{abstract}
We give a rigorous formulation of the intuitive idea that a differentiable map should be the
same thing as a {\em locally}, or {\em  infinitesimally, linear} map: just as 
a linear map respects the operations of addition and multiplication by scalars in
a vector space or module,  a {\em  locally linear map}  is defined to be a map respecting two canonical operations
$\ast$ and $\bullet$ living ``over'' its domain of definition.
These two operations are composition laws of a {\em canonical groupoid} and of a {\em scaled action category}, respectively,
fitting together into a {\em canonical double category}. 
Local linear algebra (of first order) is the study of such double categories and of their morphisms; it is 
a purely algebraic and conceptual (i.e., categorical and chart-independent) version of first order differential calculus. 
In subsequent work, the higher order theory (using higher multiple categories) will be investigated. 
\end{abstract}

\maketitle

\setcounter{tocdepth}{1}

\section*{Introduction}

The present work continues the line of investigations on general differential calculus and general
differential geometry started with \cite{BGN04, Be08, Be13, BeS, Be14}. Combining it with ideas present
in work of Nel \cite{Nel} and in synthetic differential geometry (see \cite{Ko10, MR91}), we obtain 
a purely algebraic and categorical presentation of the formal rules underlying differential calculus.
The results can be read  on two different levels:

\begin{itemize}
\item
even for ordinary (finite-dimensional, real or complex) manifolds $M$,
the construction of a canonical {\em first order difference groupoid} $M^\sett{1}$ and of a
{\em first order double category} $M^\settt{1}$ seem to be new -- indeed, they contain as a special
case a new and more conceptual construction of Connes' {\em tangent groupoid} (\cite{Co94}, II.5),
and hence our theory may be of some interest in non-commutative geometry and quantization (see, e.g., 
\cite{La}), 
\item
and these constructions open the way for  a  general, ``conceptual'',  approach to calculus and
manifolds over any commutative base ring. 
\end{itemize}

\nin
The term ``conceptual differential calculus'' is an allusion to the title of the book \cite{LaSch}, in the sense
that ``conceptual'' means ``categorical''. Most of the concepts we are going to use
(in particular, {\em double categories} and {\em double
groupoids})   go back to work of Charles Ehresmann,
see \cite{E65}; but, while Ehresmann applied them to the output of differential calculus
(i.e., to differential geometry), I shall advocate here to apply them already on the level of  the input (i.e., to the calculus
itself).

\subsection{Topological differential calculus}
 ``Usual'' differential calculus is not 
intrinsic, in the sense that it takes place in a chart domain, and not directly on a manifold:
the usual difference quotient, for a function $f:U \to W$, defined on a subset $U$ of a vector space $V$,
at a point $x \in U$
in direction $v$ and with $t\not=0$,
\begin{equation}\label{eqn:DQ}
F(x,v,t):= \frac{f(x+tv) - f(x)}{t}
\end{equation}
depends on the vector space structure, hence on a chart, and
it cannot directly be defined on a manifold.
Thus,  one  first has to develop ``calculus in vector spaces'', from which one then extracts ``invariant''
information, in order to define manifolds and structures living on them;
our work \cite{BGN04, Be08} is no exception to this rule --  in Section \ref{sec:Diff} of the present work,
we recall that approach, which we call  {\em topological differential calculus} (cf.\ \cite{Be11}).

\subsection{Conceptual differential calculus: the groupoid approach}
The path from topological calculus to the conceptual version to be presented here has been quite long,
and I refer the reader wishing to have more ample motivation and heuristic explanations
to my books and papers given in the reference list (in particular, the attempts to 
solve the problems listed in \cite{Be08b} have played an important r\^ole).
Also, it would take too much room to mention here all the work that influenced this approach, foremost 
synthetic differential geometry (see, e.g., \cite{Ko10, MR91}). 

\ssk
In a nutshell, our categorical approach can be summarized as follows: 
intuitively, we think of a manifold, or of a more general ``smooth space'', as a space $M$ that is locally, or
infinitesimally, linear. 
Seen algebraically,

\begin{itemize}
\item
 a linear space, $V$,  
is defined by two laws, $+$ and $\cdot \, $, living
\ul{\em in} the space $V$, meaning that these laws are {\em everywhere defined} on $V \times V$, resp.\
on $\K \times V$ (here, $\K$ is the base field or ring), 
\item
saying that $M$ is locally linear amounts to saying that $M$ is defined by two laws,
$\ast$ and $\bullet$, living \ul{\em over} the space $M$, in the sense that they are {\em not everywhere defined} and
 live in a certain {\em bundle,
$M^\settt{1}$, over $M \times \K$.}
\end{itemize}

\nin
More precisely, $\ast$ is a {\em groupoid law} and
$\bullet$ a {\em category law}; and the compatibility of $+$ and $\cdot$ in $V$ generalizes to the
compatibility of $\ast$ and $\bullet$, meaning that the whole structure forms a 
{\em (small) double category}.\footnote{I do not assume the reader to
have any knowledge in category theory since I am myself not a specialist in this domain -- full
definitions are given in the appendices \ref{App:Cat}, \ref{App:ActionCat}, \ref{App:Doublecat}.
I apologize in advance to  experts in category theory
for a possibly somewhat ideosyncratic presentation.}
In fact, the law $\ast$ generalizes Connes' tangent groupoid (\cite{Co94}, II.5):
as in Connes' construction, for each $t \in \K$, the fiber in $M^\settt{1}$ over $M \times \{ t \}$ is
still a groupoid; for $t=0$ we get the usual tangent bundle $TM$ (with its usual vector bundle structure),
and for invertible $t$, we get a copy of the {\em pair groupoid over $M$}.
Our construction is natural and does not proceed by taking (as in \cite{Co94}) a disjoint union of groupoids: if $M$ is a
Hausdorff manifold,  it is obvious from our construction that we get
 an {\em interpolation between the pair groupoid and the tangent bundle of $M$} (Theorem \ref{th:Manifold-one}) .

\ssk
Starting with the difference quotient
(\ref{eqn:DQ}), it is indeed quite easy to explain how to arrive at these concepts -- see Sections \ref{sec:First} and 
\ref{sec:Scalar}:
multiplying by $t$ in (\ref{eqn:DQ}), we get the notion of {\em difference factorizer} (terminology following
Nel, \cite{Nel}). Analyzing  further this concept, we realize that ``addititve'' difference factorizers are in
1:1-correspondence with morphisms of a certain $\ast$-groupoid;
if the difference factorizer is, moreover, ``homogeneous'', it corresponds to a morphism
of a certain $(\ast,\bullet)$-double category. Summing up, the formal concept corresponding  to usual 
mappings of class $C^1$ is precisely the one of morphisms of $(\ast,\bullet)$-double categories:
{\em maps of class $C^1$ are thus the same as infinitesimally linear maps.} 
This comes very close to the point of view of synthetic differential geometry (achieved there by
topos-theoretic methods). 

\subsection{Laws of class $C^1_\K$}
Once these observations are made, we can generalize {\em $C^1$-maps} by {\em $C^1$-laws over
an arbitrary  base ring $\K$}: 
given a set-map $f:U \to W$, a {\em $C^1_\K$-law over $f$} is a morphism of double categories
$f^\settt{1}:U^\settt{1} \to W^\settt{1}$  (Section  \ref{sec:Law}).
Although $f^\settt{1}$ need not be uniquely determined by $f$ (just like a polynomial need not be determined
by its underlying polynomial map), we think of $f^\settt{1}$ as a sort of {\em derivative} of $f$. Indeed,
a differentiable map, in the usual sense, gives rise  to a (unique) continuous law.
We prove also that every polynomial law (in the general sense of \cite{Ro63})  is a law of class $C^1_\K$
(Theorem \ref{th:polylaw}).
Using general principles on the construction of
manifolds (Appendix \ref{App:A}), we explain in Section \ref{sec:mflaws} how these concepts
 carry over to {\em general $C^1$-manifold laws over $\K$}.
For instance, the {\em Jordan geometries} defined over an arbitrary commutative base ring $\K$ (\cite{Be13b})
are manifold laws.

\tableofcontents

\subsection{Tentative description of the sequel}
In Part II of this work \cite{Bexy}, we shall define {\em laws of class $C^n$} and
 prove that they are morphisms of {$2n$-fold categories}, 
and, in particular, of  {\em $n$-fold groupoids}, and study their structure. 
Topics for further work include:  revisiting notions of differential geometry
(and of synthetic differential geometry), in particular, {\em connection theory}, 
{\em Lie groups} and {\em symmetric spaces}, from the groupoid viewpoint
(in this context, the paper \cite{BB} is highly relevant);
a conceptual version of the {\em simplicial approach} presented in \cite{Be13}, 
and, finally, the very intriguing topic of possible {\em non-commutativity of the
base ring}: as noted in \cite{Be08b}, Problem 8, it is possible to
develop most of the first order theory without assuming commutativity of
$\K$. Indeed, it turns out, in the present work, that commutativity of $\K$ does not enter
before dealing with bilinear maps (Section \ref{sec:Law}); the construction of the first order double
category $M^\settt{1}$ goes through for possibly non-commutative rings.
This makes it clear that commutativity becomes crucial for second and higher order calculus, but not
before; and it seems very likely that a careful analysis of this situation may lead to a new conceptual
foundation of {\em super-calculus} (cf.\ \cite{Be08b}, Problem 9).

\subsection{Notation and conventions}\label{ssec:Notation}
Throughout, the letter 
$\K$ denotes a {\em base ring with unit $1$}. Unless otherwise stated, this ring is not equipped with a topology
and not assumed to be commutative.  All $\K$-modules $V,W,\ldots$ are assumed to be
{\em right} $\K$-modules. 
By definition, a {\em linear set} is a pair $(U,V)$, where $V$ is a $\K$-module and $U \subset V$ a
non-empty subset. The linear set $(\{ 0 \},\{ 0\})$ will be denoted by $0$ (``terminal object'').
Informally, by {\em local linear algebra} we mean the theory of linear sets, their prolongations and
morphisms, as developed in this work.

\begin{acknowledgment}
I would like to thank M\'elanie Bertelson for illuminating discussions concerning the paper
\cite{BB} and for explaining to me the usefulness of the groupoid concept in differential geometry, and
Ronnie Brown  for helpful comments on double categories and double groupoids. 
I also thank Anders Kock for his critical and constructive remarks on V1 of this work.
\end{acknowledgment}

\section{Difference factorizers and differential calculus}\label{sec:Diff}

\subsection{Difference factorizer}
In order to get rid of the division by the scalar $t$ in the  difference quotient (\ref{eqn:DQ}),
we define, following a terminology used by Nel \cite{Nel}:

\begin{definition}[Difference factorizer] \label{def:difffactorizer}
Let $U \subset V$ be a linear set (cf.\ conventions above). 
As in \cite{BGN04, Be08, Be11}, we define its {\em first prolongation} by
$$
\boxed{ U^{[1]}  := U^{[1]}_\K := U^{[1]}_{V,\K}:=  \{ (x,v,t) \in V \times V \times \K \mid \, x \in U, \, x+vt \in U \} } .
$$
The {\em non-singular part} of the first prolongation is the set where $t$ is invertible:
\begin{equation}\label{eqn:extended2} 
( U^{[1]})^\times := \{ (x,v,t) \in U^{[1]}  \mid \, t \in \K^\times  \} ,
\end{equation}
For a map $f$ from $U$ to a $\K$-module $W$, a {\em difference factorizer for $f$}  is a map
 $$
 \boxed{
 f^{[1]}:U^{[1]}  \to W 
 \mbox{ such that }
 \forall (x,v,t)\in U^{[1]} :
 f(x+vt) - f(x) =  f^{[1]}(x,v,t)\cdot  t  } \, .
 $$
 \end{definition}
 

\nin 
When $(x,v,t)$ belongs to the non-singular part, then 
 $f^{[1]}(x,v,t)$ is necessarily given by (\ref{eqn:DQ}), and the proof of the following
relations is straightforward:  $\forall s,t \in \K^\times$, 
\begin{equation}\label{eqn:a}
f^{[1]} (x,0,t) = 0, 
\end{equation}
\begin{equation}\label{eqn:b}
f^{[1]}(x,v+v',t) = f^{[1]} (x+vt, v',t) + f^{[1]}(x,v,t) ,
\end{equation}
\begin{equation} \label{eqn:c}
f^{[1]}(x,vs,t) =  f^{[1]} (x,v,st) \, s  ,
\end{equation}
and, if $g$ and $f$ are composable, then, for $t\in \K^\times$, 
\begin{equation} \label{eqn:d}
(g \circ f)^{[1]} (x,v,t) = g^{[1]} \bigl( f(x), f^{[1]}(x,v,t) , t \bigr) .
\end{equation}
Now, difference factorizers are not unique -- e.g., the values for $t=0$ are not determined by the 
condition.\footnote{
If $\K$ is a field, then $t=0$ is the only ``exceptional'' value; but it will be very important
to allow more general rings. Note that the case of an integral domain, like $\K=\Z$, and free $\K$-modules, 
behaves much like the case of a field, in the sense that $f^{[1]}(x,v,t)$ is unique for all $t\not=0$.} 
Differential calculus is a means to assign a well-defined value when $t=0$.
We recall briefly the main ideas, following \cite{BGN04}; as in \cite{Be11} we will call
this theory
{\em topological differential calculus}:

\subsection{Topological differential calculus} \label{ssec:topdiff}

\begin{definition}
The {\em assumptions of topological differential calculus} are:
 $\K$ is a {\em topological ring having a dense unit group} $\K^\times$, and $V,W$
are {\em topological $\K$-modules}. Maps $f:U \to W$ are assumed to be defined on open subsets $U \subset V$.
\end{definition}

\begin{definition}
We say that {\em $f:U \to W$ is of class $C^1_\K$} if $f$ admits a {\em continuous} difference factorizer $f^{[1]}$.
Because of density of $\K^\times$ in $\K$, such a difference factorizer is unique, if it exists, and hence we can define
the {\em first differential of $f$ at $x$} by
$$
df(x) v := f^{[1]}(x,v,0).
$$
\end{definition}

The philosophy of differential calculus can be put with the words of G.\ W.\ Leibniz (quoted in the introduction of
\cite{Be08}):  ``The rules of the finite continue to hold in
the infinite'' --  properties valid for difference factorizers and invertible scalars $t$ continue to hold for ``singular'' scalars,
in particular for the most singular value $t=0$. 
For instance, by density and continuity, identities (\ref{eqn:a}) -- (\ref{eqn:d}), continue to hold  for $t=0$, proving 
the  ``usual'' properties of the differential, {\em linearity} and {\em chain rule}:
\begin{align}
df(x) (v+v') & = df(x)v + df(x)v', \cr
df(x) (vs) & =(df(x)v)s,  \cr
d(g\circ f)(x) v & = dg(f(x)) \bigl( df(x)v \bigr)
\end{align}
which then allow to define manifolds having a linear tangent bundle, and so on.
What we need for such ``invariant'' constructions is essentially only a ``functorial'' rule like the cain rule;
this permits to define bundles, carrying structure according to what is preserved under coordinate changes
(cf.\ Appendix \ref{App:A}).

\section{The first order difference groupoid and its morphisms}\label{sec:First}

\subsection{The first order difference groupoid}

\begin{definition}
If $(U,V)$ is a linear set, we shall henceforth use the notation\footnote{The deeper reason for this change of notation will
only appear at second order, when iterating the constructions: in \cite{BGN04}, the index $[2]$ is used, but in Part II we shall rather use
$\{ 1 , 2 \}$.}
$$
U^\sett{1}:= U^{[1]}, \qquad (U^\sett{1})^\times := (U^{[1]})^\times ,
$$
and we define two surjections called ``source'' and ``target''
\begin{align*}
\pi_0 : U^\sett{1}  \to U \times \K, &  \quad  (x,v;t) \mapsto (x;t)
\cr
\pi_1 :  U^\sett{1}  \to U \times \K, &  \quad  (x,v; t) \mapsto (x +  v t ; t)
\end{align*}
and one injection called ``zero section'' or ``unit section'' 
$$
z : U \times \K \to U^\sett{1}, \quad (x,t) \mapsto (x,0,t) .
$$
\end{definition}

\nin Obviously, $z$ is a {\em bisection} of the projections, i.e.,  $\pi_0 \circ z = \id_U = \pi_1 \circ z$.

\begin{lemma}\label{la:factorizer1}
Assume given a map $f:U \to W$.
Then there is a 1:1-correspondence between
difference factorizers of $f$ and 
maps $f^{\{ 1 \}} : U^\sett{1} \to W^\sett{1}$ that commute with source and target maps
 and coincide with $f$ on the base, in the sense that
$$
\mbox{ for } \sigma = 0,1: \,
\pi_\sigma \circ f^\sett{1} = (f \times \id_\K) \circ \pi_\sigma :\qquad 
\begin{matrix} U^\sett{1} & \stackrel {f^\sett{1}}{ \longrightarrow} & W^\sett{1} \cr
\downdownarrows & & \downdownarrows \cr
U \times \K & \stackrel{f \times \id_\K}{\longrightarrow} & W \times \K
\end{matrix}
$$
Namely, the difference factorizer $f^{[1]}$ corresponds to the map
$$
f^{ \{ 1 \} } :U^\sett{1} \to W^\sett{1}, \quad  (x,v,t) \mapsto  \Bigl( f(x) , f^{[1]}(x,v,t); t \Bigr) .
$$
Moreover, $f^{[1]}$ satisfies condition (\ref{eqn:a}) if, and only if, 
 $f^{\{ 1\} }$ commutes with $z$ in the sense that
$f^\sett{1} \circ z = z \circ (f \times \id_\K)$.
\end{lemma}

\begin{proof}
Assume $f^{\{ 1 \}}  : U^\sett{1} \to W^\sett{1}$ is a map. The condition
$\pi_0 \circ f^{ \{ 1 \} } =(f \times \id_\K)  \circ \pi_0$  is equivalent to the existence of a map $F$ such that,
for all $(x,v,t) \in U^\sett{1}$,
$$
f^{ \{ 1 \} } (x,v,t) = \bigl(
f(x) , F(x,v; t); t \bigr),
$$
 and then the condition 
$\pi_1 \circ f^{ \{ 1 \} } =(f \times \id_\K)  \circ \pi_1$ is equivalent to
$$
f(x) + F(x,v;t) \cdot t  = f(x+v t),
$$
for all $(x,v,t) \in U^\sett{1}$. Thus $f^{\{ 1\} }$ commutes with projections iff
 $f^{[1]} (x,v,t):=F(x,v,t)$ is a difference factorizer.
 Finally, the condition $f^\sett{1} (x,0,t) = (f(x),0,t)$ is equivalent to
  $F(x,0,t)=0$, for all $t$, that is, (\ref{eqn:a}). 
\end{proof}

\begin{definition}
Given a map $f:U \to W$, a map $f^\sett{1}:U^\sett{1} \to W^\sett{1}$ as in the lemma will be called
a {\em map  over $f$}, and $f$ will be called the {\em base map} of $f^\sett{1}$.
\end{definition}

\begin{definition} For $a=(x,v,t),a'=(x',v',t') \in U^\sett{1}$ such that
$\pi_1 (a) = \pi_0(a')$ (so $t=t'$ and $x' = x+vt$),  we define
$$
\boxed{
a' \ast a := 
(x+vt , v',t) \ast (x,v,t) := (x , v' + v, t) }.
$$
Note that $a' \ast a$ belongs again to $U^\sett{1}$. Indeed, $x+(v+v')t= x' + v't$, that is, 
\begin{equation} \label{eqn:pia}
\pi_0(a' \ast a) = \pi_0(a), \qquad
\pi_1(a' \ast a)= \pi_1(a').
\end{equation}
\end{definition}

\begin{theorem}\label{th:FirstGroupoid}
The data
$\bigl(\pi_0,\pi_1:U^\sett{1} \downdownarrows U \times \K, z, \ast \bigr)$
define a {\em groupoid}.
This groupoid is a bundle of groupoids over $\K$, i.e.,
for every fixed value of $t \in \K$,
we have a groupoid $(\pi : U_t \downdownarrows U, z, \ast)$,
$$
U_t := \{ (x,v) \mid \, x \in U, x+vt \in U \} \downdownarrows U, \qquad
(x',v') \ast (x,v) = (x,v'+v) ,
$$
$$
U_t \times_U U_t = \{ ((x',v'),(x,v) \in U_t \times U_t \mid \, x' = x+vt, \, t' = t \} .
$$
\end{theorem}

\begin{proof} We check the defining properties of a groupoid (Appendix \ref{App:Cat}):
let $a=(x,v,t)$, $a=(x',v',t')$, $a''=(x'',v'',t'')$.
As noted above, (\ref{eqn:pia}) holds.
To check associativity,
$$
(x'',v'',t'') \ast ((x',v',t') \ast (x,v,t)) =
(x'',v'',t'') \ast (x,v'+v,t) = (x,v'' + (v'+v) , t),
$$
 $$
((x'',v'',t'') \ast (x',v',t')) \ast (x,v,t) = (x',v''+v',t') \ast (x,v,t) = (x,(v''+v')+v,t) ,
$$
and hence associativity of $\ast$ follows from associativity of addition in $(V,+)$.
Next,
$$
(x,0,t) \ast (x',v,t) = (x', 0 + v,t) = (x',v,t), 
$$
$$
(x,v,t) \ast (x,0,t) = (x, v + 0, t) = (x,v,t),
$$
hence $z(x,t)$ is a unit for $\ast$. 
We show that $(x+vt,-v,t)$ is an inverse of $(x,v,t)$:
$$
(x,v,t) \ast (x+ vt,-v,t) = (x+ vt,0,t),
$$
$$
(x+vt ,-v,t) \ast (x,v,t) = (x,0,t) .
$$
It is obvious from these formulae that, for any fixed $t$, we get again a groupoid. 
\end{proof}

\begin{definition}
The groupoid $(\pi_0,\pi_1,U^\sett{1} \downdownarrows U\times \K,z,\ast)$
defined by the theorem is called the {\em first order difference groupoid of $U$}.
The symbol $U^\sett{1}$ will often be used both to denote the morphism set 
and the groupoid itself, and we use $(U^\sett{1})^\times$ for the groupoid with underlying morphism
set $(U^\sett{1})^\times$ defined by (\ref{eqn:extended2}).
\end{definition}

\begin{theorem}[The groupoids $U_0$ and $U_1$, and Connes' tangent groupoid] \label{la:isos} $ $

\begin{enumerate}
\item
The groupoid $U_0$ is  a ``group bundle'' $TU:=U_0 = U \times V$ over the base $U$, with fiber the group
$(V,+)$.
\item 
The groupoid $U_1$
is isomorphic to the {\em pair groupoid $U \times U$ over $U$}.
(See Example \ref{ex:pairgroupoid}: composition on $U \times U$ is
$(x,y) \circ (y,z) = (x,z)$.)
\item
More generally, for every $t \in \K^\times$, the groupoid $U_t$ is isomorphic to the pair groupoid over $U$, 
via 
\begin{align*}
\Phi_t :
U_t &  \to U \times U, \quad
(x,v)   \mapsto (y,x) := (x+vt,x) .
\end{align*}
The {\em non-singular groupoid}
 $(U^\sett{1})^\times$ is, via the map $(x,v,t) \mapsto (x+vt,x,t)$, isomorphic to the direct product 
$U\times U \times \K^\times$ of the pair groupoid of $U$ with the trivial groupoid of $\K^\times$.
\item
If $\K$ is a field, then $U^\sett{1}$ is isomorphic to a disjoint union of groupoids
$$
TU \, \,  \dot\cup  \, \, [(U \times U) \times \K^\times] .
$$
\end{enumerate} 
\end{theorem}

\begin{proof} (1)  is obvious from the formulae.  To prove (3), note that 
$\pr_2 \circ \Phi_t(x,v) = x = \pi_0(x,v)$, 
$\pr_1 \circ \Phi_t(x,v)= x+vt = \pi_1(x,v)$, and
\begin{align*}
\Phi_t ((x+tv,v') \ast (x,v)) & =  \Phi_t(x, v+v') = (x+(v'+v)t,x) 
\cr
& = (x+v't + vt, x+vt) \circ (x+vt,x) =
\Phi_t(x+tv,v') \circ \Phi_t(x,v).
\end{align*}
Concerning units, note that $\Phi_t (x,0) = (x,x)$,
proving that $\Phi_t$ is a morphism. It is bijective since
$(y,x) \mapsto (x, (y-x)t\inv)$ is an inverse map. 
Finally, (2) is the  case $t=1$ of (3), and (4) follows from (3) since, for a field, $\K = \{ 0 \} \, \dot\cup \,  \K^\times$.
\end{proof}

For $\K=\R$, the construction from Part (4) corresponds to Connes' construction of the
{\em tangent groupoid}, \cite{Co94}, II.5. Note that our construction gives, when
 $V \cong \R^n$, ready-made the topology
defined by Connes in loc.cit., p.\ 103. 

\begin{remark}[Pregroupoids]
By forgetting the units, $U^\sett{1}$ defines (just like any groupoid) a 
 {\em pregroupoid} (see  \ref{rk:pregroupoids}):
the ternary product $[a'',a',a] = a'' \ast (a')\inv \ast a$ is explicitly given by  $[a'',a',a] = (x, v'' - v' + v,t)$. 
\end{remark}

\subsection{Morphisms of  first order difference groupoids}

\begin{theorem}\label{th:morph1} 
Given a map $f:U \to W$,  there is a 1:1-correspondence between
\begin{enumerate}
\item
difference factorizers of $f$ satisfying Condition (\ref{eqn:b}) for all $t \in \K$, 
\item
morphisms of groupoids  $f^{\{ 1 \}} : U^\sett{1} \to W^\sett{1}$ over $f$.
\end{enumerate}
\end{theorem}

\begin{proof}
In view of Lemma \ref{la:factorizer1}, 
it only remains to show that $f^\sett{1}$ preserves $\ast$ iff
$f^{[1]}$ satisfies (\ref{eqn:b}).
We compute
\begin{align*}
f^{\{ 1\} } ((x',v',t') \ast (x,v,t) ) &= 
f^{\{ 1\} } (x,v' + v,t) =  (f(x), f^{[1]} (x,v'+v,t),t) 
\cr
f^{\{ 1\} }(x',v',t') \ast f^{\{ 1\} }(x,v,t)  &=
(f(x),f^{[1]} (x+vt,v',t),t)  \ast (f(x), f^{[1]}(x,v,t),t) 
\cr
&=
\bigl(f(x), f^{[1]}(x+vt,v',t) + f^{[1]}(x,v,t), t \bigr) \, .
\end{align*}
Thus equality holds iff (\ref{eqn:b}) holds for $f^{[1]}$.
Finally, note that Condition (\ref{eqn:a}) follows from (\ref{eqn:b}) by taking $v=0$ there.
\end{proof}

\begin{example}\label{ex:lin}
Every $\K$-linear map $f:V \to W$ gives rise to a morphism
$$
f^\sett{1}:V^\sett{1} = V^2 \times \K \to W^\sett{1} = W^2 \times \K, \quad
(x,v;t) \mapsto (f(x),f(v);t) .
$$
Indeed, since, by linearity, $f(x+vt) - f(x) = f(v) t$, the map $f^{[1]}(x,v;t)=f(v)$
 is a difference factorizer, and it satisfies
(\ref{eqn:a}) since $f$ is linear. 
\end{example}

\begin{example}
If $f(x)=\alpha x + b$ is an affine map, then
$f^\sett{1}(x,v;t) = (\alpha x + b , \alpha v;t)$ is a groupoid morphism.
\end{example}

\begin{theorem}[General morphisms]\label{th:morph!}
Assume given two maps $f:U \to W$ and $\varphi:\K \to \K'$.
Then the pair of maps
\begin{align*}
U^\sett{1}_\K \to W^\sett{1}_{\K'}, & \quad 
(x,v;t) \mapsto \bigl( f(x), F(x,v;t);\varphi(t) \bigr), \cr
U \times \K \to W \times \K', & \quad  \bigl(x,t) \mapsto (f(x),\varphi(t) \bigr)
\end{align*}
is a groupoid morphism if, and only if, $F$ is a {\em $\varphi$-twisted difference factorizer}, i.e.,
$$
\forall (x,v,t) \in U^\sett{1}:\qquad
f(x) + F(x,v;t) \cdot \varphi(t) = f(x + vt) 
$$
which satisfies, whenever defined, the condition corresponding to  (\ref{eqn:b}):
$$
F(x + vt;v',t) + F(x,v;t) = F(x,v'+v;t) .
$$
\end{theorem}

\begin{proof}
By the arguments given in the proof of Lemma \ref{la:factorizer1},  compatibility with $\pi_0$ is equivalent
to the existence of a map $F$ as in the defining formula from the theorem, and compatibility with $\pi_1$ then
is equivalent to saying that $F$ is a $\varphi$-twisted difference factorizer of $f$.
As in the preceding proof it is seen that compatibility with $\ast$ then amounts to the last condition
stated in the theorem. 
\end{proof} 

\begin{definition}
We say that a groupoid morphism $U^\sett{1}\to W^\sett{1}$ is {\em of the first kind}, or: {\em spacial},
if it is of the above form with $\K=\K'$ and $\varphi = \id_\K$, and {\em of the second kind}, or: {\em internal}, 
if it is of the above form with $f=\id_U$. 
\end{definition} 

\begin{example}\label{ex:scaling}
For $f =\id$, and $F(v,t):=F(x,v;t)$ independent of $x$, the conditions read
$$
F(v,t) \cdot \varphi(t) = vt, \qquad F(v',t) + F(x,t) = F(v'+v,t).
$$
For instance, taking, for  $s \in \K^\times$ fixed, $\varphi(t) := st$ and $F(v,t) := v s\inv$, the conditions are satisfied
(giving rise to {\em scaling automorphisms}, see next chapter).
\end{example}

\begin{example}
If $\K=\C$ and $U = V = W = \C^n$, then complex conjugation
$f(z) =\overline z$, $F(z,v;t) = \overline v$, $\varphi(t) = \overline t$ defines a groupoid automorphism.
More generally, every ring automorphism $\varphi$ of $\K$ together with 
a $\varphi$-conjugate linear map $f$ gives rise, in the same way, to a groupoid morphism
$(x,v;t) \mapsto (f(x),f(v);\varphi(t))$.
\end{example}

\subsection{The topological case, and first difference groupoid of a manifold}
Recall from Subsection \ref{ssec:topdiff}
the framework of topological differential calculus. In this case, the preceding results carry over
to the manifold level without any difficulties:

\begin{theorem}\label{th:Manifold-one} 
Assume $\K$ is a topological ring with dense unit group, 
$V,W$ Hausdorff topological $\K$-modules and $M,N$ are Hausdorff $C^1_\K$-manifolds
 modelled on $V$, resp.\ on $W$.
\begin{enumerate}
\item
A  $C^1_\K$-map $f:U \to W$ induces a morphism
of groupoids $f^\sett{1} :U^\sett{1} \to W^\sett{1}$.
\item
To the manifold $M$ we can
associate a bi-bundle $\pi_0,\pi_1: M^\sett{1} \downdownarrows M \times \K$,
carrying a canonical continuous groupoid structure.
\item
 The groupoid $M^\sett{1}$  is a bundle of groupoids over $\K$,   and hence,
for all $t \in \K$,  the fiber $M_t$ over $t$  is  a continuous groupoid with object set $M$. 
\item
The groupoid  $M_0$ is the usual tangent bundle
(additive group bundle), and $M_1$ is isomorphic to the pair groupoid over $M$. 
More generally, the groupoid $(M^\sett{1})^\times$ lying over $\K^\times$ is naturally isomorphic to  the direct product
of the pair groupoid of $M$ with the trivial groupoid of $\K^\times$. 
If $\K$ is a field, then $M^\sett{1}$ is the disjoint union of that groupoid with the tangent bundle. 
\item
Let  $f:M \to N$ be a map.
Then $f$ is of class $C^1_\K$ if, and only if, it extends to a continuous morphism of groupoids
$f^\sett{1} : M^\sett{1} \to N^\sett{1}$.
Put differently:
there is a 1:1-correspondence between $C^1$-maps $M\to N$ and continuous groupoid morphisms
$M^\sett{1} \to N^\sett{1}$ of the first kind. 
\item
There is a natural homeomorphism of bundles over $M \times N \times \K$
$$
(M \times N)^\sett{1} \quad \cong \quad M^\sett{1} \times_\K N^\sett{1}.
$$
In other words, for all $t \in \K$, we have $(M \times N)_t \cong M_t \times N_t$.
\end{enumerate}
\end{theorem}

\begin{proof}
(1) If $f$ is of class $C^1$, then it admits a continuous difference factorizer $f^{[1]}$.
As said in Section \ref{sec:Diff}, such a difference factorizer satisfies (\ref{eqn:a}) and (\ref{eqn:b}), and hence, by
Theorem \ref{th:morph1}, it  induces a morphism of groupoids $f^\sett{1}$.

\ssk
(2) 
Let us first describe the construction of the set $M^\sett{1}$: via Theorem \ref{th:reconstruct}, $M$ is described by ``local
data'' $(V_{ij},\phi_{ij})_{(i,j)\in I^2}$. By definition of a $C^1$-manifold, the transition maps $\phi_{ij}$ are $C^1$, hence we get
local data $((V_{ij})^\sett{1}, (\phi_{ij})^\sett{1})_{(i,j)\in I^2}$.
 Again by Theorem \ref{th:reconstruct}, such data define a manifold which we
denote by $M^\sett{1}$. By the same theorem, the natural projections $(V_{ij})^\sett{1} \to V_{ij} \times \K$ define projections
$\pi_\sigma: M^\sett{1} \to M \times \K$. In the same way we get the unit section $M\times \K \to M^\sett{1}$.
The products $\ast_i$ defined for each $(U_i)^\sett{1}$ are  compatible with transition maps, and hence coincide over intersections
$U_{ij} = U_i \cap U_j$. We have to show that these groupoid structures fit together  
and define a groupoid law $\ast : M^\sett{1} \times_{M\times \K} M^\sett{1} \to M^\sett{1}$.
To this end, consider $a,b \in M^\sett{1}$ such that $u:=\pi_1(a)=\pi_0(b)$.
Let $x:= \pi_0(a)$ and $y:=\pi_1(b) \in M$.
By Lemma \ref{la:handy}, we can find a chart domain $U \subset M$ containing $u,x$ and $y$. Now define
$a \ast b \in U^\sett{1}$ with respect to this chart.
As explained above, this does not depend on the chart, and we are done.

\ssk
(3) -- (6):
This now follows from (2) and  the corresponding local 
statements in Theorem \ref{th:FirstGroupoid} and  in Lemma \ref{la:isos}.
Note that the argument from (2), using Lemma \ref{la:handy} and the Hausdorff property,\footnote{ If we drop the 
Hausdorff assumption, then the same arguments show that $M_1$ is an open neighborhood of the diagonal
in $M \times M$; this open neighborhood will in general not be a groupoid, but a {\em local groupoid}, in the sense of \cite{Ko97}.}
shows that $M_1 \cong M \times M$ (pair groupoid).
For the proof of (6), the local version is, 
for chart domains $U\subset V$ and $S \subset W$, 
\begin{align*}
(U \times S)^\sett{1} & =
\{ (x,u;y,v;t) \in 
 (U \times S)  \times (V \times W) \times \K \mid 
 (x,y) + t (u,v) \in U \times S \}
 \cr 
 &=
 \{ (x,u;y,v;t) \mid (x,u,t) \in U^{[1]}, (y,v,t) \in S^{[1]} \} = U^{[1]} \times_\K S^{[1]} ,
\end{align*}
and this naturally carries over to the level of manifolds.
\end{proof}

\section{Scalar action, and the double category}\label{sec:Scalar}


\subsection{The scaling morphisms}
We have seen that the first order difference groupoid $U^\sett{1}$ takes account of the additive aspect (\ref{eqn:b})
of tangent maps. 
Now let us deal with multiplicative aspects, i.e., with the homogeniety condition 

\ssk \nin
(\ref{eqn:c})
$\qquad \qquad \qquad \qquad \qquad  f^{[1]} (x,vs,t) =  f^{[1]}(x,v,st) \cdot s$.

\begin{theorem}[The scaling morphisms]\label{th:scal}
Fix a couple of scalars $(s,t)  \in \K^2$.
Then there is a morphism of groupoids  $U_{st} \to U_t$,  given by
\begin{align*}
\phi_{s,t} : U_{st} \to U_t, & \, \,
(x,v;st) \mapsto (x,vs; t) , \qquad
\mbox{ with base map }  \id_U.
\end{align*} 
A groupoid morphism of the type $f^\sett{1}:U^\sett{1}\to W^\sett{1}$ commutes with
all morphisms of the type $\phi_{s,t}$ if, and only if, its difference factorizer 
satisfies relation (\ref{eqn:c}). 
If $\K$ is a topological ring and $M$ a Hausdorff $C^1_\K$-manifold, then the morphisms $\phi_{s,t}$
carry over to globally defined continuous groupoid morphisms   $M_{st}\to M_t$.
\end{theorem}

\begin{proof}
Clearly, we have $\phi_{s,t}\circ \pi_0 = \pi_0 \circ \phi_{s,t}$, and the condition
$\phi_{s,t}\circ \pi_1 = \pi_1 \circ \phi_{s,t}$
holds since
$x + v (st) = x + (vs)t$.
The condition  $\phi_{s,t} (a' \ast  a) = \phi_{s,t}(a') \ast \phi_{s,t}(a)$ is equivalent to
$$
(v' + v) st = v' st + vst,
$$
that is, to distributivity of the $\K$-action on $V$. 
Finally, $\phi_{s,t}(x,0,st) = (x,0,t)$, so units are preseved.
Thus  $\phi_{s,t}$ is a morphism.
Given $f:U \to W$, we compute
\begin{align*}
f^\sett{1} \circ \phi_{s,t} (x,v;st) &= (f(x), f^{[1]} (x,vs;t),t) , \cr
\phi_{s,t} \circ f^\sett{1} (x,v;st) &= (f(x), f^{[1]} (x,v;st) s, t),
\end{align*}
and the last claims follow. 
(Note that $\K$ need not be commutative for all this.) 
\end{proof}

\begin{remark}
When $s$ is invertible, we get the {\em scaling automorphism} of $M^\sett{1}$, given by
$(x,v;t) \mapsto (x,vs;s\inv t)$ (see Example \ref{ex:scaling}). For $s=-1$, we get the important automorphism 
$(x,v;t) \mapsto (x,-v;-t)$.
\end{remark}

We interprete $\phi_{s,t}$ as a {\em scaling morphism}, where $s$ is the scalar acting; morphisms
with different scaling level $t$ on target spaces have to be distinguished. 
It is quite remarkable that such structure fits together with the groupoid structure from the preceding
chapter into a {\em double category}, which we will define next.

\subsection{The first order double category}
We will define a (small) double category
\begin{equation}\label{eqn:diagram}
\begin{matrix}
C_{11}  & \rightrightarrows & C_{01}  \cr
 \downdownarrows & & \downdownarrows   
 \cr
C_{10} & \rightrightarrows & C_{00} 
 \end{matrix} :
 \qquad \qquad 
\begin{matrix}
U^\settt{1}   &  \stackrel \pi \rightrightarrows & U \times \K \times \K  &  \cr
\partial   \downdownarrows \phantom{\pi}  & & \downdownarrows   \partial
&  \cr
U^\sett{1}    &\stackrel \pi  \rightrightarrows & U \times \K 
 \end{matrix}
 \end{equation}
(see Appendix \ref{App:Doublecat} for definitions), as follows: 
for a linear set $(U,V)$, define its {\em first double prolongation} by
\begin{align}
U^\settt{1} & := \{ (x,v;s,t) \in V^2 \times \K^2  \mid \, x \in U, \, x +  v st \in U \}  \cr
& = \{ (x,v;s,t) \in V^2 \times \K^2  \mid \, (x,v;st) \in U^\sett{1} \} .
\end{align}
This comes
ready-made with  the following projections (the last two have already been defined):
\begin{align*}
\partial_0 : U^\settt{1}  \to  U^\sett{1} , &  \quad  (x,v; s,t) \mapsto (x,v;st)
\cr
\partial_1 : U^\settt{1}  \to U^\sett{1}, &  \quad  (x,v; s,t) \mapsto (x , vs ;t)
\cr
\partial_0 : U \times \K^2  \to U \times \K, &   \quad  (x; s, t) \mapsto (x;st)
\cr
\partial_1 : U \times \K^2  \to U \times \K, &   \quad  (x; s,t) \mapsto (x ;t)
\cr
\pi_0 : U^\settt{1}  \to U \times \K^2, &  \quad  (x,v; s,t) \mapsto (x;s,t)
\cr
\pi_1 : U^\settt{1}  \to U \times \K^2, &  \quad  (x,v; s,t) \mapsto (x + vst ;s,t)
\cr
\pi_0 :  U^\sett{1}  \to U \times \K, &   \quad  (x,v; t) \mapsto (x;t)
\cr
\pi_1 : U^\sett{1}  \to U \times \K, &   \quad  (x,v; t) \mapsto (x + vt ;t)
\end{align*}

\begin{lemma}
For $i , j \in \{ 0,1 \}$:
$\boxed{
\partial_i \circ \pi_j = \pi_j \circ \partial_i : U^\settt{1} \to U \times \K
}$
\end{lemma}

\begin{proof} 
By direct computation,
\begin{align*}
\partial_1 \pi_0 (x,v;s,t) & = (x,t) = \pi_0 \partial_1 (x,v;s,t),
\cr
\partial_0 \pi_0(x,v;s,t) & = (x,st) = \pi_0 \partial_0 (x,v;s,t)
\cr
\partial_0 \pi_1 (x,v;s,t) & = (x+vst,st) = \pi_1 \partial_0 (x,v;s,t)
\cr
\partial_1 \pi_1 (x,v;s,t) & = (x+vst,t) = \pi_1 \partial_1 (x,v;s,t)
\end{align*}
\end{proof}

\nin Next  define ``unit (resp.\ zero) sections''  
\begin{align*}
z_\pi :U \times \K \to U^\sett{1}, & \quad (x;t) \mapsto (x,0;t) 
\cr
z_\partial:U \times \K \to U \times \K^2, & \quad (x;t) \mapsto (x;1,t) 
\cr
z_\partial : U^\sett{1}  \to  U^\settt{1}, & \quad (x,v;t) \mapsto (x,v;1,t) 
\cr
z_\pi: U \times \K^2 \to  U^\settt{1}, & \quad (x;s,t) \mapsto (x,0;s,t) 
\end{align*}
It is immediately checked that
$
\boxed{z_\pi  \circ z_\partial = z_\partial \circ z_\pi:U \times \K \to U^\settt{1} : \, (x,t) \mapsto (x,0;1,t) }$:
$$
\begin{matrix}
U^\settt{1}   &  \leftarrow & U \times \K^2   \cr
  \uparrow \phantom{\pi}  & & \uparrow  
 \cr
U^\sett{1}    & \leftarrow & U \times \K 
 \end{matrix}
$$

\begin{lemma}
The maps $z$ are bisections of the projections defined above, that is, for $i=0,1$, 
$\boxed{\partial_i \circ z_\partial = \id, \, \pi_i \circ z_\pi = \id}$.
\end{lemma}

\begin{proof}
Immediate, since $0$ appears in the definition of $z_\pi$ and $1$ in the one of $z_\partial$.
\end{proof}

\begin{lemma} \label{la:a}   For $i=0,1$, we have
$\boxed{ \partial_i \circ z_\pi  = z_\pi \circ \partial_i: U \times \K^2  \to U^\sett{1} }$.
\end{lemma}

\begin{proof} For $i=0$,
$
z_\pi \partial_0(x;s,t) = z_\pi(x;st) = (x,0;st) = 
\partial_0(x,0;s,t) = \partial_0 z_\pi(x;s,t) 
$
and similarly for $i=1$.
\end{proof}

\begin{lemma} For $i=0,1$, we have  $\boxed{
\pi_i \circ z_\partial  = z_\partial  \circ \pi_i : U^\sett{1} \to U \times \K^2 }$
\end{lemma}

\begin{proof}
$
z_\partial \pi_1(x,v;s) =z_\partial(x+vs;s) =  (x+vs;1,s) = \pi_1( x,v;1,s) =  \pi_1 z_\partial (x,v;s)
$
for $i=1$. Similarly for $i=0$.
\end{proof}

\nin Now we define composition of morphisms. In the following formulae, we assume that
$a = (x,v;s,t), a'= (x',v';s',t') \in U^\settt{1}$. 
The two compositions $\ast$  are ``additive'' and the two compositions $\bullet$  are ``multiplicative'':

\begin{enumerate}
\item
if $\pi_1(a) =\pi_0(a')$ (so $s'=s, t'=t, x' = x+ vst$) , we define
$a' \ast a  \in U^\settt{1}$:
$$
\boxed{ a' \ast a  = (x',v';s,t) \ast (x,v;s,t)   = (x + vst ,v';s,t) \ast (x,v;s,t) 
 = (x , v+v' ; s,t) } .
$$
\item
For $(x,v;t), (x',v',t') \in U^\sett{1}$  such that $\pi_1(x,v,t) = \pi_0(x',v',t')$, (so $t'=t$, $x'=x+vt$), as
in the preceding section,
$$
\boxed{  (x' ,v';t) \ast (x,v;t)  = (x , v+v' ; t) } .
$$
\item
If  $\partial_1(a) = \partial_0(a')$ (so $x=x'$, $ v' = vs$ and $t=s't'$), 
then define $a' \bullet a$ 
$$
\boxed{ a' \bullet a  =
(x,v' ;s',t') \bullet (x, v;s ,t ) 
= (x,vs,s',t') \bullet (x,v,s,s't') =
 (x,v; ss' , t' )
} .
$$
\item
If $\partial_1 (x;s,t)=\partial_0(x'; s', t')$ (so $x'=x$ and $t=s't'$), let  
$$
\boxed{ 
(x ; s', t') \bullet (x; s , t ) = (x ; ss' ,t' )
} . 
$$
\end{enumerate}


\begin{theorem}[First order double category]\label{th:doublecat} $ $

\begin{enumerate}
\item
The data $(U^\settt{1}, U^\sett{1}, U \times \K^2,U \times \K, \pi , \partial , z , \ast , \bullet)$ define a
double category which  we denote by $\bfU^\settt{1}$ and indicate by a diagram of the form (\ref{eqn:diagram}).
\item
Morphisms of double categories
 $\bff: \bfU^\settt{1} \to \bfW^\settt{1}$ which are trivial on $\K$
are in 1:1-correspondence with maps $f:U\to W$ together with a difference  factorizer $f^{[1]}$ satisfying
(\ref{eqn:b}) and (\ref{eqn:c}).
\item
The unique map $U \to 0$ induces  a canonical morphism of the $\bullet$-category to
the left action category of $(\K,\cdot)$:
 $$
\begin{matrix}
U^\settt{1}   &  \stackrel \pi \rightrightarrows & U \times \K \times \K  &  \to & \K \times \K \cr
{ }_\partial   \downdownarrows \phantom{\pi}  & & \downdownarrows   { }_\partial
&  & \downdownarrows  \cr
U^\sett{1}    &\stackrel \pi  \rightrightarrows & U \times \K & \to & \K 
 \end{matrix}
$$
\item
The maps $j(x,v;s,t):=(x+vst,-v;t,s)$, resp.\ $j(x,v;t):=(x+vt,-v;t)$, $j(x;s,t)=(x,s,t)$, and $j(x;t)=(x;t)$, define
an isomorphism of double categories
$(\ast,\bullet) \to (\ast^{op},\bullet)$.
\item
The inverse image of the left action category $((\K^\times)^2,\K^\times)$ under the projections from Item (3)
forms a double groupoid, denoted by $(\bfU^\settt{1})^\times$, which is isomorphic to the double groupoid
given by the direct product of categories
$U \times U$ (pair groupoid) and $\K^\times \times \K^\times$ (pair groupoid).
\item
If $\K$ is a topological ring and $M$ a Hausdorff $C^1$-manifold, then there is a continuous
double category $M^\settt{1}$ over $M \times \K$, and statements analoguous to those of Theorem
\ref{th:Manifold-one}  hold. 
Continuous morphisms $M^\settt{1} \to N^\settt{1}$ which are trivial on $\K$ are in bijection with maps
$f:M \to N$ of class $C^1_\K$. 
\end{enumerate}
\end{theorem}

\begin{proof}
(1) 
We check that properties (1) -- (9) from Theorem \ref{th:smalldoublecats}
hold.
We have already checked the compatibility conditions for target and source projections and for
the unit sections (Lemmas  above), and we have seen in the preceding section that
$(U^\sett{1},\ast)$ is a category.
Similar computations show  that $(U^\settt{1},\ast)$ is a category, too.
For any fixed $x$,  $(U^\settt{1},\partial,\bullet)$ corresponds to the {\em scaled action category} from
Lemma \ref{la:scaledaction}, Appendix \ref{App:ActionCat} (with $S$ the monoid $(\K,\cdot)$), and hence
we have $\bullet$-categories.
Let us show that projections $\pi,\partial$ are morphisms between the respective categories.
We write $a'=(x',v';s',t'), a=(x,v;s,t)$, and when writing compositions $a' \bullet a$ and
$a' \ast a$, it is understood that these compositions are defined.
\begin{eqnarray*}
\pi_0(a' \bullet a) &= &\pi_0 (x,v;ss',t' ) = (x;ss',t') = (x',s',t') \bullet (x,s,t) = \pi_0(a') \bullet \pi_0(a)
\cr
\pi_1(a' \bullet a) & = &\pi_1(x,v;ss',t') = (x + vss't' ;ss',t') \cr
& = &     (x + v' s't' ;s',t') \bullet (x+vst ;s,t) =   
\pi_1(a') \bullet \pi_1(a)
\cr
\partial_0(a' \ast a) & = & \partial_0 (x,v'+v;s,t) = (x,v'+v;st) =
\partial_0(a') \ast \partial_0(a)
\cr
\partial_1(a' \ast a) & = &  \partial_1 (x,v'+v;s,t) = (x,(v'+v)s;t) =
(x,v's + vs;s) = 
\partial_1(a') \ast \partial_1(a)
\end{eqnarray*}
In the last line we  used  distributivity in $V$.
Next, the bisections $z$ are  functors:
$
z_\partial(a' \ast a) = z_\partial(a') \ast z_\partial(a) ,
z_\pi (b' \bullet b) = z_\pi(b') \bullet z_\pi(b').
$
Indeed,
\begin{align*}
(x,v'+v;1,t) &= (x,v';1,t) \ast (x,v;1,t), \cr
(x,0; ss',t) &= (x,0;s',t') \bullet (x,0;s,t).
\end{align*}
Finally,
let us prove the {\em interchange law} (\ref{eqn:interchange}):  
\begin{eqnarray*}
& & \bigl( (x',v';s',t') \ast (x,v;s,t) \bigr) \bullet 
\bigl(
(y',w';p',q') \ast (y,w;p,q) \bigr) 
\cr
&=& (x,v+v';s,t) \bullet (y,w+w';p',q') 
\cr
&=& (y,w+w';ps,t)
\cr
&= &(y',w'; p's', t' ) \ast (y,w;ps, t)  
\cr
&=&
\bigl( (x',v';s',t' ) \bullet (y',w';p',q') \bigr) \ast \bigl( (x,v;s,t) \bullet (y,w;p,q) \bigr) 
\end{eqnarray*}
This proves that $U^\settt{1}$ is a double category.

\ssk
(2) Let $\eff$ be a morphism, that is, a double functor from $\bfU$ to $\bfW$, and assume that
it is trivial on $\K$.
Denote by $f:U \to W$ the corresponding map on the base and by
$f^\sett{1}:U^\sett{1}\to W^\sett{1}$ and
$f^\settt{1}:U^\settt{1} \to W^\settt{1}$ the corresponding maps. 
As in the proof of Theorem \ref{th:morph1} and Lemma \ref{la:factorizer1},
we get $f^\sett{1}(x,v;t) = (f(x), f^{[1]}(x,v,t),t)$ with a difference factorizer $f^{[1]}$ satisfying
(\ref{eqn:a}). From compatiblity with $\pi$ we gt 
$$
f^\settt{1} (x,v;s,t) = \bigl(
f(x), F(x,v;s,t) ;s,t \bigr)
$$
with some map $F:U^\settt{1}\to W$. 
From $\partial_0 \circ f^\settt{1}\to f^\sett{1} \circ \partial_0$, it follows that
$F(x,v;s,t) = f^{[1]}(x,v,st)$, and from
$f^\sett{1} \circ \partial_1 = \partial_1 \circ f^\settt{1}$ we now get
$f^{[1]} (x,sv;t) = f^{[1]}(x,v;st) \, s$, that is (\ref{eqn:c}).
If these conditions hold, the property
$f^\settt{1} (a' \bullet a) = f^\settt{1}(a') \bullet f^\settt{1}(a)$
is proved without further assumptions.
Recall that (\ref{eqn:b}) corresponds to the property
$f^\sett{1} (a' \ast  a) = f^\sett{1}(a') \ast f^\sett{1}(a)$.
Finally, all computations can be reversed, so that a base map $f$ together with
a difference factorizer satisfying  (\ref{eqn:b}), (\ref{eqn:c})
defines a double functor $\bff$. 

\ssk
(3) 
This is proved by direct computation (cf.\ Lemma \ref{la:actioncat2}), or by using 
that the constant map $U \to 0$ induces a morphism (Lemma \ref{la:const}).

\ssk
(4) The map $j$ is the inversion map of the $\ast$-groupoids. The statement holds more generally for
double categories two of whose edges are groupoids; in the present case it can of course also be checked by
direct computations. 

\ssk
(5)
The trivialization map is
$$
\bfU^\times \to U^2 \times (\K^\times)^2, \quad
(x,v;s,t) \mapsto (y,x;u,t) := (x+vst,x;st,t)
$$
with inverse map 
$(y,x;u,t) \mapsto (x,v;s,t)= (x,(y-x)u\inv; ut\inv, t)$.

\ssk
(6) The same arguments as in the proof of Theorem \ref{th:Manifold-one} apply.
\end{proof}


\begin{theorem}[General morphisms]
Assume given a map $f:U\to W$ and two maps $\varphi:\K\to \K$, $\psi:\K\to \K$.
Then a map of the form
$$
U^\settt{1}\to W^\settt{1}, \quad (x,v;s,t) \mapsto (f(x), G(x,v;s,t) ; \varphi(s) , \psi (t) )
$$
is a morphism of double categories if and only if, whenever defined:
\begin{enumerate}
\item
there is a map $F$ such that
$G(x,v;s,t)= F(x,v;st)$,
\item
$f(x+vst) = f(x) + F(x,v;st) \, \varphi(s) \psi(t)$,
\item
$\varphi(st) = \varphi(s) \psi(t)$,
\item
$F(x+vt,v',t) + F(x,v,t)= F(x,v'+v,t)$,
\item
$F(x,vs;t) = F(x,v;st) \varphi (s)$.
\end{enumerate}
\end{theorem}

\begin{proof}
Similar to the proof of Theorem \ref{th:morph!}.
\end{proof}

\begin{definition}
A  morphism with $\varphi = \psi = \id_\K$ is called
{\em of the first kind (spacial)}, and a morphism with $f =\id$ is called
{\em of the second kind (internal)}.
\end{definition}

One can give examples similar to those following Theorem \ref{th:morph!}:
conjugate-linear maps define morphisms (then $\varphi = \psi$ must be a ring automorphism), and
there are  scaling automorphisms (then $\varphi(t) = \lambda t$, $\psi(t)=t$, $F(x,v,t)=v\lambda\inv$
for $\lambda \in \K^\times$).

\section{Laws of class $C^1$ over arbitrary rings}\label{sec:Law}

\subsection{Definition and first properties}
In this section we define the framework of {\em local linear algebra}:
we develop (first order) ``calculus'' over arbitrary base rings $\K$.
Just like {\em polynomial laws}  generalize polynomial maps,
{\em laws of class $C^1$}  generalize usual differentiable maps.
In Part II, laws of class $C^n$ and $C^\infty$ will be defined.


\begin{definition}[$C^1_\K$-laws]
Let $(U,V_1)$, $(W,V_2)$ be linear sets. A {\em $C^1_\K$-law between $U$ and $W$} is a morphism of the first kind
between double categories,  $\bff : \bfU^\settt{1} \to \bfW^\settt{1}$.
Thus $\bff$ is given by four set-maps
\begin{align*}
\bff^\settt{1}:U^\settt{1} \to W^\settt{1},& \qquad
\bff^\sett{1}:U^\sett{1}\to W^\sett{1}, 
\cr
  f \times \id_\K \times \id_\K:U\times \K^2 \to W\times \K^2, &
\qquad f \times \id_\K:U \times \K \to W \times \K,
\end{align*}
satisfying the conditions from Theorem \ref{th:doublecat}. Equivalently, 
$\eff$ is given by a base map $f:U \to W$ and a difference factorizer $\bff^{[1]}$ 
satisfying (\ref{eqn:b}) and (\ref{eqn:c}). 
We then say  that $\bff$ is a $C^1_\K$-law {\em over $f$}. 
Obviously, linear sets with $C^1_\K$-laws as morphisms
 form a (big) concrete category which we denote by
$\ul{C^1_\K\mbox{\emph{-linset}}}$. 
The set of all $C^1_\K$-laws from $U$ to $W$ will be denoted by
$C^1_\K(U,W)$. 
\end{definition}

\begin{definition}[Underlying $C^1_\Z$-law]
With notation as above, $\bff$ has an {\em underlying $C^1_\Z$-law} $\bff_\Z$, by restricting scalars in (\ref{eqn:b}) and (\ref{eqn:c})
to $\Z$. 
\end{definition}

\begin{example}\label{ex:underying}
If $\K$ is a topological ring and $V,W$ topological $\K$-modules, then a $C^1_\K$-map
(in the sense of topological differential calculus)  $f:U \to W$ gives rise to a law $\bff:\bfU \to \bfW$.
Indeed, the continuous difference factorizer of $f$ gives rise to a (continuous) morphism of
double categories $f^\settt{1}$, see Theorem \ref{th:doublecat}, Item (5). 
We call this the {\em law defined by $f$}. 
\end{example} 

\begin{remark}
We use the boldface letters in order to  stress  that $\bff$ is in general not uniquely determined by the base map $f$.
For instance, if $U = \{ x_0 \}$ is a singleton, then
$f$ is a constant map, whence
$f^{[1]}(x_0,v,t)=0$ for all $t \in \K^\times$, 
and the values $f^{[1]}(x_0,v,0)$ can be chosen independently of $f$.

\ssk
If there is no risk of confusion, we will occasionally switch back to the notation 
$f^\settt{1}$, $f^\sett{1}$ instead of $\bff^\settt{1}$, $\bff^\sett{1}$, and, keeping in mind that
these need not be  determined by $f$, we nevertheless think of $f^\settt{1}$  as a sort of
``first derivative of $f$''.
\end{remark}

\begin{definition}[Tangent map]
Given a $C^1_\K$-law $\bff$ with base map $f$ and difference factorizer $\bff^{[1]}$, and if $t \in \K$ is fixed,
we write
\begin{align*}
\bff_t :U_t =  \{ (x,v) \in U \times V, x+vt \in U \} \to W \times W, \quad
 (x,v)   \mapsto (f(x),\bff^{[1]}(x,v,t)) .
\end{align*}
For $t=0$, this map is called the {\em tangent map} of $\eff$, also denoted by
\begin{align*}
T\eff :=\eff_0: TU := U \times V  & \to TW := W \times W, \cr
(x,v) & \mapsto  T\eff (x)v :=  (f(x), d\eff (x) v) :=
(f(x), \bff^{[1]}(x,v;0))  .
\end{align*}
\end{definition}

\nin By definition of  composition of morphisms, we have the functorial rule
\begin{equation}\label{eqn:t-functor}
(\bfg \circ \bff)_t = \bfg_t \circ \bff_t,
\end{equation}
which for $t= 0$ is the ``chain rule''
\begin{equation}
T(\bfg \circ \bff) = T\bfg \circ T\bff .
\end{equation}
Written out in terms of the difference factorizers, Equation (\ref{eqn:t-functor}) reads
\begin{equation}
(\bfg \circ \bff)^{[1]}(x,v,t) =  \bfg^{[1]} ( f(x), \bff^{[1]}(x,v,t),t) .
\end{equation}

\begin{lemma}
Let $\bff : \bfU \to \bfW$ be a $C^1$-law and $x \in U$. Then
the differential
$d \eff (x) : V \to W$ is a $\K$-linear map.
\end{lemma}

\begin{proof}
As said above, $\bff^{[1]}$  satisfies (\ref{eqn:b}), (\ref{eqn:c}). For $t=0$, 
this yields the claim.
\end{proof}

\subsection{Constant  laws} The following will carry over to general manifolds and spaces:

\begin{definition} 
A {\em constant $C^1_\K$-law} is a law $\bff : \bfU \to \bfW$ such that
$\bff^\sett{1} = z_\pi  \circ f \circ \pi_0$.
\end{definition}

\begin{lemma}
A law $\bff$ is constant if, and only if, its difference factorizer vanishes: $\bff^{[1]}=0$.
There is a 1:1-correspondence between constant laws and constant maps:
$$
\boxed{ f^\settt{1}(x,v,s,t) = (c,0;s,t) }.
$$
\end{lemma}

\begin{proof}
The first statement is obvious from the formula defining $\bff^\sett{1}$.
Note that $\bff^{[1]} = 0$ satisfies  (\ref{eqn:a}) -- (\ref{eqn:c}), hence indeed
defines a law.  

Whenever $\bff^{[1]}=0$, the base map must be constant since
then $f(x + v) - f(x) = \bff^{[1]}(x,v,1) = 0$ whenever $x,x+v \in U$.  
Conversely, when the base map is constant, then the zero map certainly is a possible 
difference factorizer for $f$.
\end{proof}

\begin{remark}
If $\bff^{[1]}=0$, then $d\bff=0$, but the converse need not hold.
Note that, even in ``usual'' ultrametric calculus this need not be true -- cf.\ remarks in
\cite{BGN04}.
\end{remark}

\begin{lemma}\label{la:terminal}
Let $(U,V)$ be a linear set. Then there is a unique $C^1_\K$-law from $U$ to the linear set $0 =
(\{ 0\}, \{ 0\})$ (Conventions \ref{ssec:Notation}), given by
$f^\settt{1}(x,v;s,t)=(0,0;s,t)$.
\end{lemma}

\begin{proof} 
The law is induced by the constant map $U \to 0$ (cf.\
Part (3) of Th.\ \ref{th:doublecat}).
\end{proof}

\subsection{Linear laws} 
The following uses the linear structure of $V$ and $W$:

\begin{definition}
A {\em linear $C^1_\K$-law} is a $C^1_\K$-law $\bff:\bfV \to \bfW$ such that the map
$\bff^\settt{1}: V^2 \times \K^2 =V^2 \oplus \K^2 \to W^2 \oplus  \K^2$ is $\K$-linear.
\end{definition}

\begin{lemma}\label{la:const}
There is a 1:1-correspondence between $\K$-linear maps $f:V \to W$ and
linear $C^1_\K$-laws $\bff : \bfV \to \bfW$, given by
$$
\boxed{f^\settt{1} (x,v;s,t) = (f(x) , f(v); s,t) }.
$$
In particular, a linear law is uniquely determined by its base map. 
\end{lemma}

\begin{proof}
A  $\K$-linear map $f:V \to W$ gives rise to a morphism
$f^\settt{1}(x,v;s,t) := (f(x),f(v); s,t)$ (see Example \ref{ex:lin}), and obviously this map is 
$\K$-linear. 

Conversely, if $\bff$ is a linear law, then
$(f(x),0;0,0)=\bff^\settt{1}(x,0;0,0)$ is linear in $x$, hence $f$ is linear, and
and similarly
$\bff^{[1]}: V^2 \oplus \K \to W$ is also linear. Thus  we have, for all $t \in \K$,
\begin{align*}
\bff^{[1]} (x,v,t) & = \bff^{[1]} \bigl( (x,v,1) + (0,0,t-1) \bigr) = 
\bff^{[1]} (x,v,1) + \bff^{[1]} (0,0,t-1) \cr
&  = \bff^{[1]}(x,v,1)+0=  f(x+v) - f(x)  =  f(v),
\end{align*}
and hence, for a linear law, $\eff$ is uniquely determined by its base map $f$.
\end{proof}

\begin{example}
The addition map $a:V \times V=V \oplus V \to V$ is $\K$-linear, hence corresponds to the linear law,
called the {\em addition law of $V$},
$$
a^\settt{1}((x,y),(u,v);s,t) = (x+y,u+v;s,t) .
$$
For fixed $t,s$, this is  addition in $V^2$.
\end{example}

\begin{example}
The diagonal map $\delta : V \to V \times V=V \oplus V$, $x \mapsto (x,x)$ is linear. 
It corresponds to the {\em diagonal law}
$$
\delta^\settt{1}(x,v;s,t) = ((x,x),(v,v);s,t)
$$
For fixed $(s,t)$, this is the diagonal imbedding of $V^2$ in $V^4$.
\end{example}

\subsection{Bilinear laws, and algebra laws}
Two preliminary remarks:

\begin{enumerate}
\item 
Note that we have, for $M \subset V_1$ and $N \subset V_2$, 
like in Item (6) of Theorem \ref{th:Manifold-one},  a
natural isomorphism of bundles over $M \times N \times \K^2$
\begin{equation*}
(M \times N)^\settt{1} \quad \cong \quad  M^\settt{1} \times_{\K^2}  N^\settt{1}
\end{equation*}
by identifying $((x,y),(u,v);s,t)$ with $((x,u), (y,v) ; s,t)$ (which projects to $(x,y;s,t)$). 
We will use these identifications frequently.
\item
In this subsection (and in the following ones) we
have to assume that $\K$ is commutative (and then we prefer to write modules as {\em left} modules).
\end{enumerate}

\begin{theorem}\label{ex:Bilinearlaw}
Assume $f:V_1 \times V_2 \to W$ is a $\K$-bilinear map. Then the following formulae define a $C^1_\K$-law
 $\bff : \bfV_1 \times \bfV_2 \to \bfW$:
\begin{align*}
\bff^\sett{1} ((x,u),(y,v);t) & = \bigl(f(x,y),  f(x,v) + f(u,y) + t f(u,v) ; t \bigr) 
\cr
\bff^\settt{1} ((x,u),(y,v);s,t) & = \bigl(f(x,y),  f(x,v) + f(u,y) + st f(u,v) ; s,t \bigr) 
\end{align*}
\end{theorem}

\begin{proof}
Since $f$ is bilinear,
 $f((x,y)+t(u,v)) - f(x,y) = t (f(x,v) + f(u,y) + t f(u,v))$, hence
$
f^{[1]} ((x,y),(u,v),t) = f(x,v) + f(u,y) + t f(u,v)
$
is  a difference factorizer for $f$. It
satisfies (\ref{eqn:b}) and  (\ref{eqn:c}), hence $\bff$ indeed defines a $C^1$-law.
\end{proof}

\begin{definition}
A {\em bilinear law} is a law coming from a bilinear map, as in the theorem. If, moreover, $V_1 = V_2 = W$, then the law is
called an {\em algebra law}.
\end{definition}

\nin Thus,  by  definition, there is a bijection between bilinear base maps and their bilinear laws.
For any fixed $t$, the  map $f_t$  is again bilinear; however, 
$\bff^\sett{1}$, seen as a polynomial, is already of degree $3$. 
If $\bff$ is an algebra law, we often write $x \cdot y := f(x,y)$ (which does not mean that we assume the product to be 
associative), and then the formula for $f_t$ reads
\begin{equation}\label{eqn:t-prod}
(x,u) \cdot (y,v) := f_t\bigl( (x,u),(y,v) \bigr) =  (xy, xv + uy + t uv) .
\end{equation}

\begin{theorem}\label{th:pullback} 
Assume $f:V \times V \to V$, $(x,y) \mapsto f(x,y)=x \cdot y$ is a bilinear map.
For fixed $t \in \K$,
the set of composable elements in the category $(V_t,\ast)$,
\begin{equation*}
V_t \times_V V_t =
\{ ((x',v'),(x,v)) \in V_t \times V_t \mid \, x' = x+tv \}
\end{equation*}
is a subalgebra of the direct product algebra $V_t \times V_t$, and the law $\ast$ of this category,
\begin{equation*}
\alpha : V_t \times_V  V_t  \to \K_t, \quad (x',v';x,v) \mapsto (x',v') \ast (x,v) = (x,v' + v)
\end{equation*}
is a morphism of $\K$-algebras.
\end{theorem}

\begin{proof} 
The first claim follows from computing in $V_t \times V_t$
$$
((x',v');(x,v)) \cdot ((y',w');(y,w)) = ((x'y', x'w' + v'y' + t v'w'); (xy, xw+vy+tvw) )
$$
and noting that $x'y' = (x+tv)(y+tw)=xy + t (vy+xw+tvw)$.
Now we prove that $\alpha$ is a morphism of algebras:
\begin{align*}
\alpha ((x',v';x,v) \cdot (y',w';y,w)) & =
(xy, x'w' + v'y' + t v'w' +  xw+vy+tvw) \cr
& =
(xy, xw' + tvw' + v'y + tv' w + t v'w' +  xw+vy+tvw) \cr
\alpha (x',v';x,v) \cdot \alpha (y',w';y,w) & =
(x,v+v') \cdot (y,w + w') \cr
& = (xy, x (w+w') + (v+v') y + t (v+v') (w+w')) .
\end{align*}
Both terms coincide, hence $\alpha:V_t \times_V V _t \to V_t$ is an algebra morphism.
\end{proof}

\nin
The last claim of the theorem is a kind of {\em interchange law} (cf.\ equation (\ref{eqn:interchange})):
\begin{equation}
\bigl(
(x',v') \cdot (y',w') \bigr) \ast \bigl( (x,v) \cdot (y,w) \bigr) =
\bigl(
(x',v') \ast (x,v) \bigr) \cdot \bigl( (y',w') \ast (y,w) \bigr) 
\end{equation}
Consider the following commutative diagram:
\begin{equation}\label{eqn:algebra}
\begin{matrix}
V^\sett{1} & \to & 0^\sett{1} = \K \cr
\downdownarrows & & \downarrow \id_\K \cr
V \times \K & \to & 0 \times \K = \K
\end{matrix}
\end{equation}
The two vertical arrows indicate small categories (where the law $\ast$ on $0^\sett{1}$ is trivial, but the one on $V^\sett{1}$ is not),
and the two horizontal arrows indicate bundles of products $\cdot$ indexed by $\K$. 
The whole thing satisfies properties similar to the ones of a small double category, except that $\cdot$ need not be
associative or unital. If, however,  the product $\cdot$ on $V$
 is associative and unital, then $(V,\cdot)$ is a monoid, hence a small category with one
object, and the theorem implies that diagram (\ref{eqn:algebra}) defines  a small double category with products $\ast$ and $\cdot$ .
For the special case $V= \K$, with its bilinear ring product, we can identify the product on $V_t$ in terms of truncated polynomial
rings, as follows:

\begin{lemma}[The ring laws]\label{ex:Ringlaw}
Let $a: \K \times \K \to \K$ and $m:\K\times \K \to \K$ be addition and multiplication in the (commutative) ring $\K$.
 For any $t \in \K$, these define maps
$$
a_t : \K_t \times \K_t \to \K_t, \qquad m_t :\K_t \times \K_t \to \K_t.
$$
Identifying $\K_t$ with $\K^2$, the maps $a_t$ and $m_t$ define a ring structure on $\K^2$, which is
isomorphic to the ring $\K[X]/(X^2 - tX)$  with $\K$-basis $[1]$ and $[X]$.
\end{lemma}

\begin{proof}
The ring structure on $\K[X]/(X^2 - tX)$ is given by $[X^2] = t [X]$ whence
$$
(x [1] + u [X] ) (y [1] + v [X]) = xy [1] + (xv + uy + t uv) [X].
$$
Comparing with (\ref{eqn:t-prod}), we see that the bilinear product on $\K_t$ is given by the same formula,
and hence $\K_t$ is a ring, isomorphic to $\K[X]/(X^2 - tX)$.
\end{proof}

The neutral element of $\K_t$ is $e=(1,0)$. We identify $\K$ with the subalgebra $\K e$ and thus consider
$\K_t$ as $\K$-algebra. According to Theorem \ref{th:pullback},  for fixed $t \in \K$,
the set 
$
\K_t \times_\K \K_t =
\{ ((x',v'),(x,v) \in \K_t \times \K_t \mid \, x' = x+tv \}
$
is a subalgebra of the direct product algebra $\K_t \times \K_t$, and the law $\ast$ of this category,
\begin{equation*}
\ast  : \bK_t \times_\K \K_t  \to \K_t, \quad (x',v';x,v) \mapsto (x',v') \ast (x,v) = (x,v' + v)
\end{equation*}
is a morphism of $\K$-algebras.

\begin{example}[Module laws]
Left multiplication by scalars, $m_V:\K \times V \to V$, $(\lambda,v) \mapsto \lambda v$, is $\K$-bilinear, hence
gives  rise to a law 
$$
m_V^\settt{1}:\K^\settt{1} \times_{\K^2} V^\settt{1} \to V^\settt{1} .
$$ 
For any $t \in \K$, we get a
map $(m_V)_t: \K_t \times V_t \to V_t$.
Writing explicitly the formulae, one sees that this map describes  the action of the
ring $\K_t = \K[X]/(X^2 - tX)$ on the scalar extended module 
$V_t = V \otimes_\K (\K[X]/(X^2 - tX))$, given by 
$(r,s) \cdot (x,v) = (rx, rv + sx + tsv)$.  
We call this the {\em $\K$-module law of $V$}.
\end{example}

\subsection{Polynomial laws}
Informally, a polynomial law is given by a map together with all possible scalar extensions. 
We recall from \cite{Ro63} the relevant definitions (see also \cite{Lo75}, Appendix);
 the base ring $\K$ is assumed to be commutative.

\begin{definition}
Denote by $\ul{\Alg_\K}$ the (big) concrete category of unital commutative $\K$-algebras and $\ul\set$ the concrete category of sets
and mappings. 
Any $\K$-module $V$ gives rise to a functor 
$\ul V : \ul{\Alg_\K} \to \ul{\set}$ by associating to $\bA$ the scalar extended module $V_\bA =
V \otimes_\K \bA$ and to $\phi : \bA \to \bB$ the induced map $\phi_V: =\id \otimes_\K \phi:V_\bA \to V_\bB$.

\ssk
A {\em polynomial law between $V$ and $W$} is defined to be a natural transformation
$P : \ul V \to \ul W$, i.e., for every $\K$-algebra $\bA$ we have a map
$P_\bA : V_\bA \to W_\bA$, compatible with algebra morphisms $\phi:\bA \to \bB$ in the sense that
$P_\bB \circ \phi_V = \phi_W \circ P_\bA$. We say that $P_\bA:V_\bA \to W_\bA$ is a {\em scalar extension} of the
{\em base map} $P_\K:V \to W$. 
\end{definition}

\begin{theorem}[Polynomial laws are $C^1$-laws]\label{th:polylaw}
Every polynomial law $P:\ul V \to \ul W$ gives rise to a $C^1_\K$-law ${\bf P} : \bfV \to \bfW$.
More precisely, if $P$ is a polynomial law, 
letting $P_r := P_{\K[X]/(X^2 - rX)}:V_r \to W_r$,  the map $P^\settt{1}$ is obtained by
$$
P^\settt{1} (x,v;s,t):= (P_{st} (x,v) ; s,t) . 
$$
In particular, the tangent map $TP:TV \to TW$ is given by
scalar extension by {\em dual numbers} $\K[X]/(X^2)$, and the differential
$d P(x) : V \to W$, $v \mapsto P_0(x,v)$ is $\K$-linear. 
\end{theorem}

\begin{proof}
Let us prove that $P^\settt{1}$, defined as in the theorem, is a $C^1_\K$-law.
This is done by showing that all relevant maps are induced by algebra morphisms. Note first that
the two projections 
$$
\pi_0:\K_t \to \K , \, (x,v) \mapsto x, \qquad \pi_1:\K_t \to \K, \, (x,v) \mapsto x+tv
$$
are algebra morphisms, and they induce the two projections
$\pi_i := V_{\pi_i} : V_t \to V$.
Thus, from the definition of polynomial laws, we get $\pi_i \circ P_t = P_\K \circ P_i$, hence 
$$
P^{[1]} (x,v,t) := \pr_2 (P_t (x,v)) 
$$
is a difference factorizer for the base map $P_\K:V \to W$. 
Let us show that  $P^{[1]}$ satisfies (\ref{eqn:c}). 
For any $(s,t )\in \K^2$, the map $\phi_{s,t}:\K_{st} \to \K_t$, $(x,v) \mapsto (x,sv)$ is an algebra morphism, as is
immediately checked. 
It induces a map $\Phi_{s,t}:V_{st} \to V_t$, and
by definition of polynomial laws, we then have
$P_t \circ \Phi_{s,t}= \Phi_{s,t}\circ P_{ts}$.
The computation given in the proof of Theorem \ref{th:scal} shows that then (\ref{eqn:c}) holds.

\ssk
Let us  prove that $P^{[1]}$ satisfies (\ref{eqn:b}).
By Theorem \ref{th:pullback}, $\ast : \K_t \times_\K \K_t \to \K_t$ is an algebra morphism,
and this morphism induces the category law
$\ast : V_t \times_V V_t \to V_t$. 
By definition of a polynomial law, $P$ commutes with the induced maps, which means that
$P_t (a' \ast a) = P_t(a') \ast P_t(a)$, or, equivalently, that $P_t$ satisfies (\ref{eqn:b}).

\ssk
Finally, for $r=0$, $P_0$ is obtained by scalar extension with $\K_0 = \K[X]/(X^2)$.
\footnote{
The proof given here generalizes the one from \cite{Be14} proving linearity of $dP(x)$.
The proof of linearity given in 
\cite{Ro63} (and in \cite{Lo75}) is different,
using heavily the decomposition of a polynomial into homogeneous parts, which is not adapted
to the case $r \not= 0$.}
\end{proof}

\begin{remark}
Constant, linear, and bilinear $C^1$-laws obviously come from polynomial laws, in the way 
described by the theorem. 
\end{remark}

\begin{remark}[Formal laws]
In the second part of his long paper
\cite{Ro63}, Roby defines and investigates {\em formal laws} (``lois formelles''). They generalize formal
power series.  
We will show in subsequent parts of this work  that {\em formal laws are laws of class $C^\infty$}.
The underlying linear set is $(\{ 0 \}, V)$ (since $0$ is they only point where  all formal series converge).
\end{remark}

\section{Manifold laws of class $C^1$}\label{sec:mflaws}

\subsection{$C^1$-manifold laws over $\K$} 
Using the general principles from Appendix \ref{App:A}, subsets of $V$ can be glued together
by using a specified set of laws (``atlas law''):

\begin{definition}
Let $\K$ be a topological ring and $V$ a topological $\K$-module; we do not assume here that $\K^\times$ is dense in $\K$,
so in particular the discrete topology on $\K$ and $V$ is admitted.
A {\em $C^1_\K$-manifold law  modelled on $V$}, denoted by $\bf M$, is given by $C^1_\K$-laws
$(\mathbf{V},\cT_{\bf V},({\bf V}_{ij},\phi_{ij})_{(i,j) \in I^2} )$ in the sense of
Theorem \ref{th:reconstruct}, that is, base sets and base maps of these data form ``gluing data'' as described in that theorem,
and likewise for the other components of the laws. 
For each $U = V_{ij}$,
the topology on $U^\sett{1}$ shall be the initial topology with respect
to $\pi_1,\pi_0:U^\sett{1} \to U$. 
The manifold law is called {\em handy} if the base manifold $(V,(V_{ij},\phi_{ij}))$ is handy in the sense of Definition
\ref{def:handy}.

A {\em $C^1_\K$-law} $\bff$ between manifold laws $\bf M$ and $\bf N$
is given by a family of $C^1_\K$-laws $\bff_{ij}$ such that all components of the law are is an Theorem \ref{th:reconstruct}. 
\end{definition}

\begin{theorem}
Let $\bf M$ be a manifold law over $\K$, modelled on $V$.
Then there are primitive manifolds $M^\sett{1}$, $M^\settt{1}$ together with projections and  injections fitting into 
the diagram
$$
\begin{matrix}
M^\settt{1}   &  \stackrel \pi \rightrightarrows & M \times \K \times \K  & \rightarrow & \K^2 \cr
{ }_\partial   \downdownarrows \phantom{\pi}  & & \downdownarrows   { }_\partial
& & \downdownarrows \cr
M^\sett{1}    & \stackrel \pi \rightrightarrows & M \times \K & \rightarrow &  \phantom{\, . } \K \,  .
 \end{matrix} 
 $$
There are also partially defined products $\ast$ and $\bullet$ defining,
if the manifold law is handy, the structure of a double category on $M^\settt{1}$.   
If $\bf N$ is another manifold law over $\K$, modelled on $W$, then $C^1_\K$-laws $\bff$ 
between $\bf M$
and $\bf N$ are precisely the morphisms of the structure defined by projections, injections and partially defined laws
$\ast$ and $\bullet$. 
\end{theorem} 

\begin{proof}
Existence of the sets $M^\sett{1}$ and $M^\settt{1}$, as well as of the projections and injections, follows directly from the
above definition combined with Theorem \ref{th:reconstruct}. 
As in the proof of Theorems \ref{th:doublecat}  and \ref{th:Manifold-one}, the same holds for the definition of locally
defined products $\ast$ and $\bullet$. As in these proofs, the only point that needs attention is when checking if
$M^\sett{1}\times_{(M \times \K)} M^\sett{1}$ is stable under $\ast$: for this, we infer the assumption that the manifold law
is handy, as in the proof of Theorem \ref{th:Manifold-one}.
\end{proof}

\begin{remark}
In the general case (handy or not), we obtain a {\em local small double category} $M^\settt{1}$. 
To avoid technicalities, we do not give formal definitions here (see, e.g.,  \cite{Ko97}):
the groupoid law is no longer defined on all of $M^\sett{1}\times_{(M \times \K)} M^\sett{1}$, but only on an open neighborhood of
the zero section;  and similarly for $M^\settt{1}$.
\end{remark} 

\subsection{Towards more general categories of spaces}
We postpone the general theory of manifold laws to subsequent parts of this work. 
In guise of a conclusion,
let us, however, already mention that this very general category of
manifolds still has the drawbacks that the category of usual manifolds already has:
(1) the lack of {\em inverse images}, (2) it is {\em not cartesian closed}:

\begin{remark}[Inverse images]
In general, the inverse image of a small [double] category (and of a [double] groupoid) under a morphism is again
a small [double] category (resp.\ a [double] groupoid).
Therefore, if $\bff:{\bf M} \to {\bf N}$ is a morphism, and $c \in N$ a fixed element,
then the inverse image of the ``isolated point $c$'' (sub-double category ${\bf c}:= \{
(c, 0;s,t) \mid s,t \in \K \}  \cong \K^2$)  in  $\bf N$ under
$\bff$ is again a sub-double category of $\bf M$. 
Explicitly, on the chart level,
 if $\bff:U^\settt{1} \to W^\settt{1}$ is a law and $c \in W$ a fixed element, the inverse image
\begin{align}
\bff\inv({\bf c}) & = \{ (x,v;s,t) \in U^\settt{1} \mid  f^\settt{1}(x,v;s,t) = (c , 0 ;s,t) \}
\cr
& = \{ (x,v;s,t) \in U^\settt{1} \mid f(x)=c, \, f^{[1]}(x,v,st) = 0 \}
\end{align}
is a sub-double category. 
However, $\bff\inv({\bf c})$ will in general not be a manifold: it may be a ``singular space''. 
\end{remark}

\begin{remark}[Cartesian closedness]
If $M$ and $N$ are usual manifolds, then the set $C^1(M,N)$ of $C^1$-morphisms from $M$ to $N$
is in general not a manifold.
The same problem arises for any other kind of manifolds.
On the other hand, the space of mappings from $U$ to a $\K$-module $W$ is always a linear space, with pointwise
defined structure, having the additive maps as subspace (this is true for {\em commutative} groups $(W,+)$ and
fails for general groups). The following result can be interpreted by saying that the ``locally linear maps''
share this property (and the proof shows that commutativity of $(W,+)$ enters here in the same way): 
\end{remark}

\begin{theorem}[The double category structure on the set of $C^1$-laws] $ $

\begin{enumerate}
\item
The set $\Hom_\K(U^\sett{1},W^\sett{1})$ of groupoid morphisms of the first kind between $U^\sett{1} $ and $W^\sett{1}$
carries a natural groupoid structure, namely, the {\em pointwise groupoid structure} inherited from $W^\sett{1}$.
\item
If  $\K$ is commutative,
then the set  $C^1_\K(U,W)$ of $C^1$-laws from $U$ to $W$ carries a natural structure of small double category,
given by the pointwise structure.
In particular, for $U=W$ and $t=0$,
we get the natural linear structure on the {\em space of vector fields}.
\end{enumerate}
\end{theorem}

\begin{proof} (1)
Let $\bff,\bfg \in \Hom_\K(U^\sett{1},W^\sett{1})$.
The two projections are $\bff \mapsto \pi_i(\bff):= \pi_i \circ \bff$.
Assume $\pi_1(\bfg) = \pi_0(\bff)$, that is,
\begin{equation}\label{eqn:wd}
\forall  c \in U^\sett{1}: \qquad \pi_1 (\bfg (c)) = \pi_0 (\bff (c)) ,
\end{equation}
 and define a map by  ``pointwise product''
$\bff \ast \bfg: U^\sett{1} \to W^\sett{1}$, $a \mapsto \bff(a) \ast \bfg(a)$.
Because of (\ref{eqn:wd}) this is well-defined.
We show that $\bff \ast \bfg$ is again a groupoid morphism:
let $a',a \in U^\sett{1}$ such that $\pi_1(a) = \pi_0(a')$, so $a' \ast a$ is defined, hence
\begin{align*}
(\bff \ast \bfg) (a' \ast a) & = \bff (a' \ast a) \ast \bfg (a' \ast a) 
 = \bff (a') \ast \bff(a) \ast \bfg(a') \ast \bfg (a) .
\end{align*}
From (\ref{eqn:wd}) with $c = a' \ast a$ we get
$$
\pi_1 \bfg (a') =\pi_1 \bfg (a' \ast a) = \pi_0 \bff (a' \ast a) = \pi_0 \bff (a) .
$$
On the other hand, since $\bff$ is a morphism, from $\pi_1 (a) = \pi_0(a')$, it follows with (\ref{eqn:wd}),
$$
\pi_1 \bff (a) = \pi_0 \bff (a') = \pi_1 \bfg (a'),
$$
so that $\pi_0 \bff (a)= \pi_1 \bfg (a')  = \pi_1 \bff(a) = \pi_0 \bfg (a')$.
Thus both $\bff(a)$ and $\bfg(a')$ are {\em endo}morphisms of the same object.
Now, since $(W,+)$ is {\em commutative}, endomorphisms of the same object commute:
\begin{equation}
(x,v',t) \ast (x,v,t) = (x,v' + v ,t) = (x,v+v',t) = (x,v,t) \ast (x,v',t) ,
\end{equation}
and hence we get 
\begin{align*}
(\bff \ast \bfg) (a' \ast a) & = \bff (a') \ast \bfg(a') \ast \bff(a)  \ast \bfg (a) 
 = (\bff \ast \bfg)(a') \ast (\bff \ast \bfg) (a) ,
\end{align*}
proving that $\bff \ast \bfg$ is again a morphism.
Thus $\Hom_\K(U^\sett{1},W^\sett{1})$ is stable under the pointwise structures, and by general principles,
we get again a structure of the same type, that is, a groupoid.
If $\K$ is commutative, then endomorphisms in the $\bullet$-categories also commute with each other, 
so that the same arguments as above imply that the pointwise product $\bff \bullet \bfg$ belongs
to a well-defined category structure on $\Hom_\K(\bfU,\bfW)$, which together with the $\ast$-structure
defined before forms a double category, proving (2). 
\end{proof}

Summing up, there ought to be  some (big) cartesian closed
category of local double categories containing $C^1$-manifold laws and their inverse images.
This category would, then, be a good candidate for a general notion of ``space laws''.  
We come back to this issue as soon as the general $k$-th order theory is developed.

\appendix

\section{Categories, groupoids}\label{App:Cat}


\subsection{Concrete categories}
Formally, a 
\href{https://en.wikipedia.org/?title=Concrete_category}{\em concrete category} (abridged: ccat) is defined as a category together with a faithful
functor to the category of sets. For our purposes, concrete categories are just a piece of language, and 
we rather think of a concrete category as given by a certain ``type $\mathcal T$ of structure
defined on sets'':
objects are ``sets with structure of type $\mathcal T$'', 
and morphisms are  ``maps preserving structure''. We will denote such concrete categories by
underlined roman letters, for instance

\begin{itemize}
\item
the ccat $\ul{\rm Vect}_\K$ of all $\K$-vector spaces (with linear maps as morphisms),
\item
the ccat $\ul{\rm Top}$ of all topological spaces (with continuous maps),
\item
the ccat $\ul{\rm Man}_\K$ of  smooth manifolds over $\K$ (with smooth maps),
\item
the ccat $\ul{\rm Grp}$ of all groups,  its subcat $\ul{\rm Cgroup}$ of all commutative groups,
\item
the ccat $\ul{\rm Ring}$ of all rings, its subcat $\ul{\rm Field}$ of all fields,
\item
the ccat $\ul{\rm Alg}_\K$ of all (associative) $\K$-algebras, 
\item 
the ccat $\ul{\rm Cat}$ of all small cats (see below)
and its subcat $\ul{\rm Goid}$ of all groupoids,
\item
the ccat $\ul{\rm Set}$ of all sets (with arbitrary set-maps). 
\end{itemize}

\subsection{Small categories}
A {\em small category} (abridged: small cat) is given by a pair of sets $(B,M)$,
$B$ called the {\em set of objects} and $M$ called the {\em set of morphisms}, 
together with two maps ``source'' and  ``target'' $\pi_0,\pi_1:M \to B$
(we shall write $\pi_\sigma$, where $\sigma$ takes two values; instead of $0$ and $1$ one may also use the values
$s$ et $t$, or $+$ and $-$, or others), and map $z: B \to M$ ``zero section'' or ``unit section'', and finally 
 a binary
composition $g \ast f$, defined for $(g,f) \in M \times_B M$ where
\begin{equation}\label{eqn:Order}
M\times_B M := M\times_{B,\pi} M:= \{ (g,f) \in M \times M \mid \, \pi_0(g) = \pi_1(f) \} ,
\end{equation}
such that these data satisfy the following properties: 
\begin{equation} 
 \pi_1(g \ast f) = \pi_1(g), \qquad   \pi_0(g \ast f) = \pi_0(f) ,
\end{equation} 
and the law $\ast$ is {\em associative}: whenever $(h,g)$ and $(g,f)$ are in $M \times_B M$, then
\begin{equation}
(h \ast g) \ast f = h \ast (g \ast f) .
\end{equation}
Moreover,   $z: M \to B$ is a {\em bisection} (that is, $\pi_\sigma  \circ z = \id_B$ for $\sigma=0,1$), 
 such that
\begin{equation}
z(\pi_1(f)) \ast f = f, \qquad g \ast z(\pi_0(g)) =g .
\end{equation} 
The small cat will be denoted by 
$(B,M,\pi_\sigma,s,\ast)$, or, shorter,
$(M \downdownarrows B, \ast)$ or $(B,M)$. 
Especially in Part II it will be useful to write $(C_0,C_1)$ instead of $(B,M)$, and to call the disjoint union
$C := C_0 \dot\cup C_1 = B \dot\cup M$ the {\em underlying set} of the small cat.

\ssk
By  a {\em bundle of small cats} $M \rightrightarrows B \rightarrow I$ we just mean 
an indexed family $(M_i,B_i)$ of small cats indexed by a set $I$.
The total space $(M,B)$ is then again a small cat.

\subsection{Functors}
A {\em morphism between small cats} $(M \downdownarrows B,\ast)$, $(M' \downdownarrows B',\ast')$, or {\em (covariant) functor},
 is a pair of maps $(F:M \to M',f:B\to B')$ such that
\begin{equation}
f \circ \pi_0 = \pi_1' \circ F, \qquad
F \circ z = z' \circ f, \qquad
F(h) \ast' F(g) = F(h \ast g) . 
\end{equation}
It is obvious that the composition of functors is again a functor, and that the identity $\id_C$ of the underlying set
$C = M \dot\cup B$ is a functor.
Identifying $(F,f)$ with $f \dot\cup F:C \to C'$, 
small cats form a concrete category $\ul{\rm{Cat}}$.

\subsection{Groupoids} Assume $(B,M)$ is a small cat.
An element $f \in M$ is {\em invertible} if there is another one, $g$, such that
$\pi_0(g) = \pi_1(f)$ and $\pi_1(g)=\pi_0(f)$ and $g \ast f = z (\pi_0(f))$ and
$f \ast g = z(\pi_0(g))$.
By standard arguments, such a $g$ is unique. It is then called the {\em inverse of $f$}
and denoted by $g = f\inv$.
A {\em groupoid} is a small category in which {\em every} $f\in M$ is invertible.
Groupoids and functors form a concrete category $\ul{\rm Goid}$. 
For every groupoid, the {\em inversion map} $i:M \to M$, $f \mapsto f\inv$, together with $\id_B$, is an isomorphism onto the
opposite groupoid. 
See, e.g., \cite{Ma05}, for more information on groupoids. 

\ssk
A {\em group bundle} is a groupoid such that $\pi_1 = \pi_0$. Then the fibers of $\pi_\sigma$ are groups.

\begin{example}[The pair groupoid]\label{ex:pairgroupoid} 
If $A$ is a set, then
$(A \times A, \downdownarrows, A, z,\circ)$ with $\pi_0(x,y)=y$, $\pi_1(x,y)=x$,
$z(x)=(x,x)$ and
$(x,y) \circ (y,z) = (x,z)$ is a groupoid, called the {\em pair groupoid} of $A$.
The inverse of $(x,y)$ is $(y,x)$.
If $R \subset (A \times A)$, then $(R,A,z,\circ)$ is a subgroupoid of the pair groupoid if, and only if,
$R$ is an equivalence relation on $A$. 
\end{example}

\begin{example}[The anchor morphism]
For a category $(\pi:M \downdownarrows B,z,\ast)$, the {\em anchor}
\begin{equation}
\kappa : M \to B\times B, \, a \mapsto (\pi_1(a),\pi_0(a)),\qquad 
B \to B, \, a \mapsto a
\end{equation}
is a morphism of $(M,B)$ to the pair groupoid. Indeed, if $\pi_1(a) = \pi_0(a')$,
$$
\kappa (a' \ast a) = (\pi_1(a'\ast a),\pi_0(a'\ast a)) = (\pi_1(a'),\pi_0(a)) =
\kappa(a') \circ \kappa(a).
$$
\end{example}

\subsection{Opposite category,  and notation}
The {\em opposite category (resp.\ groupoid)} of a category (resp.\ groupoid)  $(\pi_0,\pi_1, M,B,\ast,z)$ 
 is $(\pi_0^{op},\pi_1^{op},M,B,\ast^{op},z)$, where
\begin{equation}
\pi_0^{op}:=\pi_1, \qquad 
\pi_1^{op}:=\pi_0, \qquad
g \ast^{op} f:= f \ast g.
\end{equation} 
Note that each groupoid is isomorphic to its opposite groupoid, via the inversion map (but a category needs not be
isomorphic to its opposite category).
A {\em contravariant functor} between categories is  a morphism to the opposite category.
We are aware that several authors use other conventions concerning notation of the product.\footnote{
See \cite{BHS11}, p. 145 Rk 6.1.1 and 
p.\ 556: ``The first notation is taken from the composition of maps and the second is more algebraic.
... we have used both...''.  Our convention, given by
 (\ref{eqn:Order}), follows the conventions that seem to be most common,
see, e.g., \url{http://ncatlab.org:8080/nlab/show/category} and
\url{http://ncatlab.org/nlab/show/opposite+category}.
}
However, once a convention is fixed, it should be kept.
%

\subsection{Pregroupoids}\label{rk:pregroupoids}
In every groupoid we may define a ternary product 
\begin{equation}
[a'',a',a]:= a'' \ast (a')\inv \ast a,
\end{equation}
 whenever
$\pi_1(a)=\pi_1(a')$ and $\pi_0(a'')=\pi_0(a')$. 
In Part II we will need the notion of  {\em pregroupoid} (cf.\ \cite{Ko10}, see also \cite{Be14b}):

\begin{definition}
A {\em pregroupoid} is given by a set $M$ together with two surjective maps $a:M \to A$, $b:M \to B$ and a
partially defined {\em ternary product map}
$$
M \times_a M \times_b M \to M, \quad (x,y,z) \mapsto [xyz] 
$$
where
$$
M \times_a M \times_b M = \{ (x,y,z) \in M^3 \mid \, a(x)=a(y), \, b(y)=b(z) \} ,
$$
such that these data satisfy
$$
\forall (x,y,z) \in M \times_a M \times_b M : \quad
a([xyz])= a(x), \, b([xyz])=b(z),
$$
and the {\em para-associative} and the {\em idempotent law} hold:

\ssk
{\rm (PA)}
$\quad [x[uvw]z]=[[xwv]uz]=[xw[vuz]]$,

{\rm (IP)}
$\quad \, [xxz] = z = [zxx]$.
\end{definition}

\nin
Note that such structure only depends on the equivalence relations of fibers defined by $a$ and $b$, hence the sets $A$ and $B$
may be eliminated from the definition by considering $a,b$ just as equivalence relations on $M$ (as done in \cite{Be14b}). 
A {\em morphism of pregroupoids} is given by a map $f:M \to M'$ sending fibers of $a$ to fibers of $a'$ and fibers of $b$ to fibers of 
$b'$ and preserving the ternary product: $f[xyz]=[fx,fy,fz]$.
Obviously, pregroupoids and their morphisms form a concrete cat $\ul{\rm Pgoid}$.

\begin{theorem}
The ccat of groupoids is equivalent to the ccat of pregroupoids together with a
fixed bisection.
\end{theorem}

\begin{proof}
This observation is due to Johnstone, cf.\ \cite{Be14b}.
\end{proof}

\begin{example}
If $A,A'$ are two sets, there is a pregroupoid 
$(A \times A', \downdownarrows,A,A',[ \, , \, , \, ])$ with
$[(x,y),(u,y),(u,v)] = (x,v)$. 
It admits a bisection if, and only if, $A$ and $A'$ are equipontentious.
See \cite{Be14b} for more on this.
\end{example}

\section{Scaled monoid action category} \label{App:ActionCat}

\subsection{The action groupoid}
If $S$ is a group and $V \times S \to V$ a right group action, then 
the following construction of the {\em action groupoid} is well-known in category theory (see, eg.,
\url{http://ncatlab.org/nlab/show/action+groupoid}).
The sets of
morphisms and objects are defined by
$(M,B) =(V \times S, V)$,
with two projections $\rho_i:V \times S \to V$, 
$$
\rho_0(v,g)=v, \qquad \rho_1(v,g) =  vg ,
$$
and product when $vg= v'$
$$
(v',g') \bullet (v,g) = (vg, g') \bullet (v,g)  = (v, gg' ) .
$$
Units are $(v,1)$, where $1$ is the unit of $S$, and
the inverse of $(v,g)$ is $(vg, g\inv)$. 
In case of a left action, we take $M=S \times V$, $B=V$ and
$(g',v') \bullet (g,v)=(g'g,v)$.

\subsection{Monoid action category}
Let $S$ be a monoid acting from the right on a set $V$, via $V \times S \to V$.
(In the main text, $S = (\K,\cdot)$ is the multiplicative monoid of a ring $(\K,+,\cdot)$, and
$V$ a $\K$-module.)
Then, of course, the preceding 
construction still works, but instead of a groupoid it  merely defines a
small  cat.

\begin{remark}\label{rk:U} 
In the preceding situation
 we may define, for any non-empty subset $U \subset V$, a subcategory
$$
M_U := \{ (v,g)\in V \times S \mid \, v \in U, vg \in U \}, \qquad
B_U:= U .
$$
Even if $S$ is a group, this need not be a groupoid (the category $M_U$ then rather belongs to the
semigroup $\{ g \in G \mid  U.g \subset U\}$, and not to a group).
\end{remark}

\subsection{The scaled action category}
The effect of the construction of the action groupoid
 is to break up the ``single object $V$'' into a collection of different objects and 
to distinguish all the isomorphisms $(v,g)$ when $v$ runs over the set of objects. 
If $S$ is not a group,  we will need a kind of refinement of the monoid action category: 
we will  distinguish various ``scaling levels'' of $v$ -- 
the couple $(v,k)$ with $k \in S$ should be seen as ``the object $v$, scaled at level $k$''.
Then each $g\in S$ gives rise to morphisms (denoted by
$(v;g,k)$) from $v$, scaled at $gk$, to $vg$, scaled at $k$. Finally, a most symmetric formulation of this  
concept is gotten when the ``scale'' $k$ lives in a space $K$ on which $S$ acts from the left (later we take $S=K$).

\begin{lemma}\label{la:scaledaction}
Assume $S$ is a monoid acting from the right on $V$ and from the left on a set $K$. 
The following data $(M,B,\partial,z,\bullet ) = (V \times S \times K,V \times K,\partial,z,\bullet)$ define a small cat, called
the {\em scaled action category}:
\begin{eqnarray*}
\partial_0: V \times S \times K \to V \times K, &
(v;s,t) \mapsto (v;st)
\cr
\partial_1 : V \times S \times K  \to V \times K, &
(v;s,t) \mapsto (vs ;t) .
\end{eqnarray*}
The composition $a' \bullet a$ for $a=(v;s,t), a'=(v';s',t')$
is defined if $\partial_1(a) = \partial_0(a')$,
so $(v',s't')=(vs,t)$, so $v'=vs$, $s't' =t$,
\begin{equation}
(v';s',t') \bullet (v;s,t) = (vs; s',t') \bullet (v;s,s't') := (v; ss',t') ,
\end{equation}
and, if $1$ denotes the unit of the monoid $S$, the unit section is defined by 
\begin{equation}
z(v;s):=
(v;  1 ,s)  .
\end{equation}
If, moreover, $S$ is a group, then the category $(M,\downdownarrows,B)$ is a groupoid.
\end{lemma}

\begin{proof}
Everything is  checked by straightforward computations. For convenience, we give some details: first, 
note that

 $\partial_0(a' \bullet a) = (v,ss't') = (v,st)  =  \partial_0(a)$

 $\partial_1(a' \bullet a) = (vss',t') = (v's', t')  =  \partial_1(a')$, 
 
\nin and associativity follows from the one of $S$:
 $$
 \bigl(
 (v'';s'',t'') \bullet (v';s',t') \bigr) \bullet (v;s,t) =
 (v'; s' s'' ,t'') \bullet (v;s,t) = (v; s s' s'' , t'')
 $$
 $$
 (v'',s'',t'') \bullet \bigl(
 (v';s',t') \bullet (v;s,t) \bigr) =
 (v'';s'',t'') \bullet (v;s s',t ) =
 (v; s s' s'', t'' ) .
 $$
The element
$
(v;  1 ,s)
$
is a unit for the categorial product:
$(v';s',t') \bullet (v;1,t) = (v;s',t')$ (note $v'=v1=v$)
and
$(v',1,t') \bullet (v;s,t) = (v;s',t) = (v;s,t)$ since $s=s't'=s'$. 
Note that the morphism $(v;s,t)$ is invertible if, and only if, $s$ is invertible in the multiplicative semigroup
of $\K$:
$(v';s,t') \bullet (v,s\inv ,t ) = (v,1,t' )$. 
\end{proof}

 \begin{remark} As above (Remark \ref{rk:U}),
for every $U\subset V$, there is  a  subcategory 
$\{ (v;s,t) \mid \, v\in U, vst \in U \}$.
\end{remark}

\begin{lemma}\label{la:actioncat2}
With notation from the preceding lemma,  the  following maps of objects and morphisms 
 define a functor from the scaled action category 
to the left action category of $S$ on $K$:
\begin{align*}
V \times S \times K \to S \times K, &  \quad  (v,s,t) \mapsto (s,t), \cr
V \times K \to K, & \quad  (v,s) \mapsto s. 
\end{align*}
 If the action $V \times S \to V$ admits a fixed point $o$, then  we also get a functor in
the other sense, via
$S \times K   \to V \times S \times K$, $(s,t) \mapsto (o;s,t)$.
\end{lemma}

\begin{proof}
The left action category of $S$ acting on itself  is given by the
morphism set $S\times S$ and object set $S$ and source and target maps
\begin{equation}
\partial_0: S\times S  \to  S, \, (s,t) \mapsto st, \qquad
\partial_1: S \times S \to  S, \, (s,t) \mapsto s,
\end{equation}
and composition, whenever $t = s't'$, 
\begin{equation}
(s',t') \bullet (s,t) = (ss',t)
\end{equation}
From these formulae it is seen that the maps given above define a functor. 
\end{proof}

\section{Small double categories, double groupoids}\label{App:Doublecat}

The following presentation follows  \cite{BrSp76} (where letters $H,V$ ``horizontal, vertical'' are used for our 
$(C_{11},C_{10},C_{01},C_{00})$). We give full details in order to prepare for the algebraic presentation of small $n$-fold
cats in Appendix B of Part II, and we are more tedious than in usual presentations,  regarding  ``size questions''.

\subsection{Small cats and groupoids of a given type}
Let $\mathcal T$ be  a concrete category. 
A {\em small cat of type $\mathcal T$} is a set $C$ carrying both a structure of type $\mathcal T$ and the structure of 
a small cat $(C_0,C_1,\pi_\sigma: C_1 \downdownarrows C_0,\ast)$, so $C = C_0 \dot\cup C_1$, such that
all structure maps of the small cat $C$ are compatible with the structure $\mathcal T$.
This includes the assumption that $C_0, C_1$ and $C_1 \times_{C_0} C_1$ also carry structures of type $\cT$
(in practice, one will often check this by first noticing that $C_1 \times C_1$ carries a structure of type $\cT$, and then that
equalizers as given by (\ref{eqn:Order}) are again of type $\cT$),
and hence it makes sense to require that $\pi_\sigma, z$ and $\ast$ are morphisms for $\mathcal T$.
A {\em morphism between two small cats of type $\mathcal  T$}, say $C$ and $C'$, is a map
$f:C \to C'$ which is a functor (for the small cat-structures) and a structure-preserving map for $\mathcal T$.
This defines a new concrete category \ul{${\mathcal T}$-{\rm Cat}}. 
In the same way the concrete category \ul{${\mathcal T}$-{\rm Goid}} of {\em groupoids ot type $\cT$} is defined.

\subsection{Small double categories}\label{ssec:DC}
A {\em (strict) small double category} (abridged: small doublecat) is a small cat of type $\ul{\rm Cat}$. 
In other words, in the preceding paragraph we take ${\mathcal T}=\ul{\rm Cat}$. This defines a concrete category
$\ul{\rm Doublecat}$.
A small doublecat is thus an algebraic structure of a certain type.
We wish to give a more explicit description, in the spirit of usual algebra: a small doublecat $C$ is, first of all, a small cat
 $C = C_1 \dot\cup C_0$ with projections $\pi_\sigma$ and product $\ast$,
  and $C_1 = C_{11} \dot\cup C_{10}$ and $C_0=C_{01}\dot\cup C_{00}$
are in turn two small cats with 4 projections all indicated by the symbol $\partial$ and two products both indicated by 
$\bullet$. Since $\pi$ restricts to two projections,
we also get 4 projections denoted by the symbol $\pi$, and similar for the unit sections, everything fitting into two
commutative diagrams:
\begin{equation}\label{eqn:doublediagram}
\begin{matrix}
C_{11}  &   \stackrel \pi  \rightrightarrows & C_{01}  \cr
 {}_\partial \downdownarrows \phantom{\pi} & &  {}_\partial \downdownarrows \phantom{\pi} 
 \cr
C_{10} &   \stackrel \pi  \rightrightarrows & C_{00}
 \end{matrix} , \qquad \qquad
\begin{matrix}
C_{11}  &   \stackrel {z_\pi}  \leftarrow & C_{01}  \cr
 z_\partial \uparrow \phantom{\pi} & &  z_\partial \uparrow \phantom{\pi} 
 \cr
C_{10} &   \stackrel {z_\pi}  \leftarrow & C_{00} 
 \end{matrix} .
\end{equation}
Using this notation, saying that every edge of the square defines a small cat and that the pairs $(\pi_\sigma,\pi_\sigma)$,
resp.\ $(\partial_\sigma,\partial_\sigma)$ are functors, amounts to the following requirements:

\begin{enumerate}
\item
$\forall i,j \in \{ 0,1 \}$:  $\partial_i \circ \pi_j = \pi_j \circ \partial_i : C_{11} \to C_{00}$, 
\item
$z_\pi \circ z_\partial = z_\partial \circ z_\pi: C_{00} \to C_{11}$, 
\item
$\forall \sigma  \in \{ 0,1 \}$:  $\partial_\sigma \circ z_\partial = \id_{C_{\sigma 0}}$, 
$\pi_\circ \circ z_\pi = \id_{C_{0\sigma}}$, 
\item
$\forall \sigma  \in \{ 0,1 \}$:  $\partial_\sigma  \circ z_\pi = z_\pi \circ \partial_\sigma: C_{01} \to C_{10}$,
$\pi_\sigma  \circ z_\partial = z_\partial \circ \pi_\sigma  :C_{10} \to C_{01}$, 
\item
the partially defined products $\ast$ and $\bullet$ are associative, 
\item
elements 
$z_\partial (u)$ are units for $\bullet$ and elements $z_\pi(v)$ are units for $\ast$,
\item
$\forall \sigma \in \{ 0,1 \}$: the pairs of maps $(\partial_\sigma,\partial_\sigma)$, $(\pi_\sigma,\pi_\sigma)$ preserve 
partially defined products: 
$\partial_\sigma (a' \ast a) = \partial_\sigma(a') \ast \partial_\sigma(a)$,
$\pi_\sigma(b' \bullet b) = \pi_\sigma(b') \bullet \pi_\sigma(b)$, 
\item
the maps
$z_\partial: (V,\ast) \to (C,\ast)$, $ z_\pi: (H,\bullet) \to (C,\bullet) $ are sections of $\partial$, resp.\ of $\pi$, and
$
z_\partial(a' \ast a) = z_\partial(a') \ast z_\partial(a)$, 
$z_\pi (b' \bullet b) = z_\pi(b') \bullet z_\pi(b)$. 
\end{enumerate}

\begin{theorem}\label{th:smalldoublecats}
Data $(C_{11},C_{10},C_{01},C_{00},\pi,\partial,z,\ast,\bullet)$ given by spaces, maps and partially defined products
define a small doublecat if, and only if, they satisfy {\rm (1) -- (8)} together with the following {\em interchange law}: 
\begin{enumerate}
\item[\rm{(9)}] whenever both sides are defined  
(which is the case iff $\pi_1(b) = \pi_0(a) = \pi_0(c) = \pi_1(d)$ and
$\partial_0(b)=\partial_1(c)=\partial_0(a)=\partial_1(d)$), then
\begin{equation*}\label{eqn:interchange}
(a \ast b) \bullet (c \ast d) = (a\bullet c) \ast (b \bullet d) .
\end{equation*}
\end{enumerate}
A morphism of small double cats is
given by four maps which define a 
functor for each of the four small cats forming the edges of (\ref{eqn:doublediagram}). 
\end{theorem}

\begin{proof}
It only remains to show that, in presence of (1) -- (8), the interchange law (9) is equivalent to saying that
$\ast: C_1 \times_{C_0} C_1 \to C_1$ defines a morphism for the $(\partial,\bullet)$-cat structure. 
Here, $C_1 \times C_1$ is the direct product of two $\bullet$-small cats, and the equalizer
(\ref{eqn:Order}) is  a small subcat since the projections $\pi_\sigma: C_1 \downdownarrows C_0$ are morphisms,
by (7). Thus the product on $C_1 \times_{C_0} C_1$ is given by 
$(a,b) \bullet (c,d):=(a \bullet c,b\bullet d)$.
Then, the map $A:=\ast :C_1 \times_{C_0} C_1 \to C_1$, $(a,b) \mapsto a \ast b$ is  a morphism for $\bullet$ iff
$$
A  ((a,b) \bullet  (c,d)) = A (a,b) \bullet A(c,d) ,
$$
that is, the interchange law 
$(a \bullet c) \ast  (b \bullet d) = (a \ast b) \bullet (c \ast d)$ holds.
\end{proof}

\begin{definition}
The four small cats forming the edges of the diagram (\ref{eqn:doublediagram}) are called the {\em edge cats} of the
small doublecat. It is obvious from the description given in the theorem that, exchanging $C_{01}$ and $C_{10}$ and
$\ast$ and $\bullet$, we get again a small doublecat, called the {\em transposed doublecat}. 
A small doublecat is called {\em edge symmetric} if it admits an isomorphism onto its transposed doublecat.
\end{definition}

\subsection{Double groupoids}
A {\em double groupoid} is a small doublecat such that each of its four edge categories is a groupoid. 
Equivalently, it is a groupoid of type $\ul{\rm Goid}$. This defines a concrete category $\ul{\rm Doublegoid}$.

\subsection{Double bundles}
A {\em double group bundle} is a double groupoid such that $\pi_0 = \pi_1$ and $\partial_0=\partial_1$.
E.g., {\em double vector bundles} (see \cite{BM92}) are double group bundles.

\subsection{Opposite double cat}
In a double cat, we may replace each of the pairs of vertical, resp.\  horizontal, cats, by
its opposite pair, and we get again a double cat.
Thus we get altogether 4 double cats:
$(\ast,\bullet)$, $(\ast^{op},\bullet)$, $(\ast,\bullet^{op})$, and $(\ast^{op},\bullet^{op})$.
In general, they are not isomorphic among each other; however, if $\ast$ or $\bullet$ belongs to a groupoid, then
inversion is an isomorpism onto $\ast^{op}$, resp.\  onto $\bullet^{op}$, and compatible with the other law.

\section{Primitive manifolds}\label{App:A}


The following is a generalization of the usual definition of 
\href{https://en.wikipedia.org/wiki/Atlas_(topology)}{atlas of a manifold}. It formalizes the idea that a space $M$ 
``is modelled on a space $V$'' (which may be a linear space, or not -- linearity of the model space is not needed
for the following).

\begin{definition}\label{def:Primmfd}
Let $V$ be a topological space that will be called {\em model space} (the topology need not be separated or non-discrete).
A {\em primitive manifold modelled on $V$} is given by 
$(M,\cT,V,(U_i,\phi_i,V_i)_{i\in I})$, where $(M,\cT)$ is a topological space,
$(U_i)_{i\in I}$ an open cover of $M$, so $M = \cup_{i\in I} U_i$ and the $U_i$ are open and non-empty, and
and $\phi_i : U_i \to V_i$ are homeomorphisms onto open sets $V_i \subset V$. 
We then also say that ${\mathcal A} = ( U_i,  \phi_i,V_i)_{i\in I}$
is an {\em atlas on $M$ with model space $V$}.
The atlas is called {\em maximal} if it contains all compatible charts (defined as usual in differential geometry, see
e.g., \cite{Hu94}).
The primitive manifold is called  a {\em Hausdorff manifold} if its topology is Hausdorff and if the atlas is maximal. 

\ssk
\nin A {\em morphism of primitive manifolds} $(M, \cT, {\mathcal A})$, $(M', \cT', {\mathcal A}')$
is a continuous map  $f:M\to M'$. Thus primitive manifolds form a concrete category $\ul{\rm Pman}$.
\end{definition}

\begin{remark}
One may suppress the topology $\cT$ from this definition by considering on $M$ the {\em atlas topology}, which is the coarsest
topology such that chart domains are open  (topology generated by the $U_i$).
This is the point of view taken in a first version of this work;
however, including the topology $\cT$ as additional datum gives more freedom and is closer to usual definitions.
\end{remark}

\begin{definition}\label{def:handy}
An atlas $\mathcal A$, and the primitive manifold $M$, are called {\em handy} if, for any finite collection
$x_1,\ldots,x_n \in M$, there exists a
chart $(U_i,\phi_i)$ such that $x_1,\ldots,x_n  \in U_i$. 
\end{definition}

\begin{lemma}\label{la:handy}
A Hausdorff manifold modelled on a  topological group $V$ is handy. 
\end{lemma}

\begin{proof}  Let $n=2$.
If $x_1=x_2$, there is nothing to show. If $x_1\not= x_2$, choose disjoint  charts $(U',\phi')$ around $x_1$ 
and $(U'',\phi'')$ around $x_2$.
 By shrinking chart domains if necessary and using translations on $V$, we may assume
that $V':=\phi'(U')$ and $V'':=\phi''(U'')$ are disjoint. But then the disjoint union $(U' \cup U'', \phi' \cup \phi'', V' \cup V'')$
is a chart (by maximality of the atlas)
with the required properties (note that connectedness is not required for chart domains), and we are done.  
For $n>2$, one proceeds in the same way. 
\end{proof}

Given an atlas, we let for $(i,j) \in I^2$, 
\begin{equation}
U_{ij}:= U_i \cap U_j \subset M , \qquad V_{ij}:= \phi_j(U_{ij}) \subset V_j , 
\end{equation}
and the {\em transition maps} belonging to the atlas are defined by
\begin{equation}
 \phi_{ij}:= \phi_i\circ{\phi_j^{-1}} \vert_{ V_{ji} }:V_{ji}\to V_{ij} .
\end{equation}
They are homeomorphisms satisfying the {\em cocycle relations}
\begin{align}\label{eqn:cocycle}
 \phi_{ii}=\id_{V_i} \quad &  \mbox{and} \quad    \phi_{ij} \circ \phi_{jk}= \phi_{ik}:  V_{kji} \to V_{ijk},
 \cr 
 & \mbox{where}  \quad V_{abc}=\phi_a(U_a \cap U_b \cap U_c) = \phi_{ab}(V_{cb} \cap V_{ab}).
\end{align}

\begin{theorem}[Reconstruction from local data]\label{th:reconstruct}
The data of a primitive manifold $(M,\cT,{\mathcal A})$ are equivalent to the data
$(V,\cT_V,(V_{ij},\phi_{ij})_{(i,j) \in I^2})$, where 
\begin{itemize}
\item $I$ is a (discrete) index set,
\item $(V,\cT_V)$ is a topological space (the model space),
\item  $ V_{ij} \subset V$,  for $(i,j) \in I^2$, are open subsets,
\item $\phi_{ij} :V_{ji}\to V_{ij}$ are homeomorphisms satisfying the cocycle relations
{\rm (\ref{eqn:cocycle})}.
\end{itemize}
Morphisms of primitive manifolds then are the same as families of continuous maps
\begin{align*} 
f_{ij}: V_{jj} \to V_{ii}', &  \qquad (i,j) \subset I' \times I, \mbox{  such that }
f_{k \ell} = \phi_{ki}' \circ f_{ij} \circ \phi_{j \ell} : V_{\ell j} \to V_{ki}.
\end{align*}
\end{theorem}

\begin{proof}
Given a primitive manifold $M$, the data $(V,\cT_V,(V_{ij},\phi_{ij})_{(i,j) \in I^2})$ are defined as above, and if 
$f:M\to M'$ is a morphism, we let $f_{ij} := \phi_i' \circ f \circ (\phi_j)^{-1}$. 

Conversely,  given $(V,\cT_V,(V_{ij},\phi_{ij})_{(i,j) \in I^2})$,
 define $M$  to be the quotient  $M:=S/\sim$, where 
$S:=\{(i,x)|x\in V_{ii} \}\subset I\times V$ with respect to the equivalence relation
 $(i,x)\sim(j,y)$ if and only if $(\phi_{ij})(y)=x$. 
 We then put 
$V_i:=V_{ii}$, $U_i:=\{[(i,x)],x\in V_i\}\subset M$ and $\phi_i:U_i\to V_i, [(i,x)]\mapsto x$.
The topology on $M$ is defined to be the topology generated by all
$(\phi_i)^{-1}(X)$ with $X$ open in $V_i$ and $i\in I$. 
Moreover, given a family $f_{ij}$ as in the theorem, the map
$$
f:M \to M', \quad [(j,x)] \mapsto [(i,f_{ij}(x))]
$$
is well-defined and continuous. 
All properties
are  now checked in a straightforward way; we omit the details (cf.\  \cite{Hu94}, Section 5.4.3).
\end{proof}

\begin{example}
If all $U_{ij}$ are empty for $i \not= j$, then $M$ is just the disjoint union of the sets $V_i:=V_{ii}$, and there are no
transition conditions. In particular, when $\vert I \vert = 1$, we see that
every open subset $U \subset V$ is a manifold, with a single chart. 
\end{example}

\begin{remark}
The preceding theorem is used to ``globalize'' local functorial constructions: if such a construction transforms local data, as
described in the theorem, into other such local data, then, again by the theorem, such a construction carries over to the manifold
level. For instance,
in topological differential calculus, this construction permits to define the {\em tangent bundle} $TM$ of a
manifold $M$, and, much more generally, the {\em Weil bundle} $FM$ of $M$ for any {\em Weil functor}
$F$ (see \cite{BeS, Be14}). In the present work, it is applied to the local construction $U \mapsto U^\sett{1}$.
\end{remark}


\begin{thebibliography}{9999999}




  \bibitem[BB11]{BB}
 Bertelson, M., and P.\ Bieliavsky, ``Affine connections and symmetry jets'',  preprint 2011, 
 \url{http://arxiv.org/abs/1103.2300}
 

\bibitem[Be08]{Be08} 
Bertram, W., {\em
Differential Geometry, Lie Groups and Symmetric Spaces over General 
Base Fields 
and Rings},  Memoirs of the AMS,  {\bf 192},  900 (2008). 
       arXiv: \url{http://arxiv.org/abs/math/0502168} 
       %
\bibitem[Be08b]{Be08b}
Bertram, W., 
 ``Difference Problems and Differential Problems.''
  In : Contemporary Geometry and Topology, p. 73 - 86 (Proceedings, Cluj University Press 2008),      
     \url{http://arxiv.org/abs/0712.0321}       

\bibitem[Be11]{Be11}
Bertram, W., {\em Calcul diff\'erential topologique \'el\'ementaire}, Calvage et Mounet, Paris 2011

       
\bibitem[Be13]{Be13} 
Bertram, W., 
``Simplicial differential calculus, divided differences, and construction of Weil functors'',
       Forum Mathematicum {\bf 25} (1) (2013), 19--47.
 \url{http://arxiv.org/abs/1009.2354}
 %
        
\bibitem[Be14]{Be14}
Bertram, W., ``Weil Spaces and Weil-Lie Groups'', 
\url{http://arxiv.org/abs/1402.2619}

\bibitem[Be14b]{Be14b}
Bertram, W., ``Universal associative geometry'', 
\url{http://arxiv.org/abs/1406.1692}

 \bibitem[Be14c]{Be13b} 
Bertram, W.,   ``Jordan Geometries -- an approach by Inversions''
      Journal of Lie Theory, {\bf 24} (4), (2014), 1067--1113,                                                                                
                         \url{http://arxiv.org/abs/1308.5888}

\bibitem[Be$\rho \tau$]{Bexy}
Bertram, W., ``Conceptual Differential Calculus. II: Higher order local linear algebra'', in preparation


\bibitem[BGN04]{BGN04} 
Bertram, W., H. Gloeckner and K.-H. Neeb, 
``Differential Calculus, Manifolds and Lie Groups over Arbitrary Infinite Fields'',  
Expo. Math. 22 (2004), 213 --282.
\url{http://arxiv.org/abs/math/0303300}


 \bibitem[BeS14]{BeS}
 Bertram, W, and A.\ Souvay,
 ``A general construction of Weil functors'', 
 Cahiers Top.\ et G\'eom.\ Diff.\ Cat\'egoriques {\bf LV}, Fasc.\ 4, 267 -- 313 (2014), 
  \url{http://arxiv.org/abs/1111.2463}
 
 
 \bibitem[BH81]{BH81}
 Brown, R., and P.J.\ Higgins, ``The algebra of cubes'',
 J. Pure Appl.\ Algebra {\bf 21} (19811), 223 -- 260 ;
 revised version  \url{http://pages.bangor.ac.uk/~mas010/pdffiles/algcubes.pdf}
 
 \bibitem[BHS11]{BHS11}
 Brown, R., Higgins, P., and R.\ Sivera,
 {\it Nonabelian Algebraic Topology},
 EMS Math.\ Tracts {\bf 15}, EMS, Z\"urich 2011
 
 \bibitem[BM92]{BM92}
 Brown, R., and K.\ Mackenzie,
 ``Determination of a double Lie groupoid by its core diagram'',
 J.\ Pure and Appl.\ Algebra {\bf 80} (1992), 237 -- 272
 
 
 \bibitem[BrSp76]{BrSp76}
 Brown, R., and C.\ Spencer,
 ``Double groupoids and crossed modules'',
 Cahiers top.\ g\'eo.\ diff.\ cat.\ {\bf 17} (1976), 343 -- 362
 
 \bibitem[Co94]{Co94}
 Connes, A., {\em Noncommative Geometry},
 Academic Press, San Diego 1994
 
 
\bibitem[E65]{E65}
Ehresmann, Ch., {\it Cat\'egories et structures}, Dunod, Paris 1965.

\bibitem[Hu94]{Hu94}
Husemoller, D., {\it Fibre Bundles} (3rd ed.), Springer, New York 1994.

\bibitem[Ko97]{Ko97}
Kock, A., {\it Extension theory for local groupoids}, \url{http://home.math.au.dk/kock/locgp.pdf}.

\bibitem[Ko10]{Ko10}
Kock, A., {\it Synthetic Geometry of Manifolds}, Cambridge Tracts in Mathematics {\bf 180}, Cambridge
2010



\bibitem[La01]{La}
Landsmann, P.A.,  ``Quantization and the tangent groupoid'', Proceedings,
Constanta 2001, 
\url{http://arxiv.org/pdf/math-ph/0208004v2.pdf}

\bibitem[LaSch09]{LaSch}
Lawvere, W., and S.\ Schanuel,
{\it  Conceptual Mathematics: A First Introduction to Categories},
 Cambridge University Press, Cambridge 2009.
 
 
\bibitem[Lo75]{Lo75}
Loos, O., {\it Jordan Pairs}, Springer LNM {\bf 450}, Berlin 1975.

\bibitem[Ma05]{Ma05}
Mackenzie, K., {\em 
General Theory of Lie Groupoids and Lie Algebroids},
Cambridge University Press, London Mathematical Society Lecture Note Series {\bf 213},
2005. 


\bibitem[MR91]{MR91}
Moerdijk, I., and G.E.\ Reyes, 
{\em Models for Smooth Infinitesimal Analysis}, Springer, New York 1991.

\bibitem[Nel88]{Nel}
Nel, L.D. ``Categorical differential calculus for infinite dimensional spaces'', 
Cahiers g\'eom.\ top.\ diff.\ cat\'egoriques {\bf 29} (1988),  257 -- 286.

\bibitem[Ro63]{Ro63}
Roby, N., ``Lois polynomes et lois formelles en th\'eorie des modules'',
Ann.\ Sci.\ E.N.S., {\bf 80} (1963), 213 -- 348.





\end{thebibliography}
\end{document}